\documentclass[10pt]{amsart}

\pdfoutput=1

\usepackage{mathrsfs, amsmath, amsthm, amssymb,}%
\usepackage{tikz}%
\usepackage[linktocpage]{hyperref}%
\usepackage[all]{xy}%
\xyoption{2cell}%
\UseAllTwocells%

\hypersetup{
 pdftitle={Ind-abelian categories and quasi-coherent sheaves},
 pdfauthor={Daniel Schaeppi},
 pdfkeywords={Tannaka duality} {algebraic stacks} {weakly Tannakian categories}
}

\DeclareMathOperator{\id}{id}
\DeclareMathOperator{\el}{el}
\DeclareMathOperator{\op}{op}
\DeclareMathOperator{\Lan}{Lan}
\DeclareMathOperator{\Mod}{\mathbf{Mod}}

\DeclareMathOperator{\fp}{fp}

\DeclareMathOperator{\Spec}{Spec}

\DeclareMathOperator{\Sh}{Sh}
\DeclareMathOperator{\Rex}{\mathbf{Rex}}

\DeclareMathOperator{\Ind}{Ind}

\DeclareMathOperator{\Aff}{\mathbf{Aff}}

\DeclareMathOperator{\Coh}{\mathbf{Coh}}
\DeclareMathOperator{\QCoh}{\mathbf{QCoh}}

\DeclareMathOperator{\fpqc}{\mathit{fpqc}}

\DeclareMathOperator{\SymPsMon}{\mathbf{SymPsMon}}

\DeclareMathOperator{\colim}{colim}

\DeclareMathOperator{\Lex}{\mathbf{Lex}}

\DeclareMathOperator{\Cat}{\mathbf{Cat}}

\DeclareMathOperator{\Comod}{\mathbf{Comod}}

\newcommand{\ca}[1]{\mathscr{#1}}

\newcommand{\Prs}[1]{\mathcal{P}\ca{#1}}

\newcommand{\ten}[1]{\mathop{{\otimes}_{#1}}}

\newcommand{\pb}[1]{\mathop{{\times}_{#1}}}

\theoremstyle{plain}
\newtheorem{thm}{Theorem}[section]
\newtheorem{prop}[thm]{Proposition}
\newtheorem{lemma}[thm]{Lemma}
\newtheorem{cor}[thm]{Corollary}

\theoremstyle{definition}
\newtheorem{example}[thm]{Example}
\newtheorem{rmk}[thm]{Remark}
\newtheorem{dfn}[thm]{Definition}

\newtheoremstyle{citing}{}{}{\itshape}{}{\bfseries}{.}{ }{\thmnote{#3}}
\theoremstyle{citing}

\newtheoremstyle{citingdfn}{}{}{}{}{\bfseries}{.}{ }{\thmnote{#3}}
\theoremstyle{citingdfn}

\numberwithin{equation}{section}


\author{Daniel Sch\"appi}
\title{Ind-abelian categories and quasi-coherent sheaves}

\thanks{The author gratefully acknowledges support from a 2012 Endeavour Research Fellowship from the Australian Department of Education, Employment and Workplace Relations.}

\begin{document}

\begin{abstract}
 We study the question when a category of ind-objects is abelian. Our answer allows a further generalization of the notion of weakly Tannakian categories introduced by the author. As an application we show that, under suitable conditions, the category of coherent sheaves on the product of two schemes with the resolution property is given by the Deligne tensor product of the categories of coherent sheaves of the two factors. To do this we prove that the class of quasi-compact and semi-separated schemes with the resolution property is closed under fiber products.
\end{abstract}

\maketitle

\tableofcontents

\section{Introduction}

\subsection{Ind-abelian categories}
 It is a well-known fact that the category $\Ind(\ca{A})$ of ind-objects of an abelian category $\ca{A}$ is again abelian. On the other hand, for any commutative ring $R$, the category of ind-objects of the category $\Mod_R^{\fp}$ of finitely presentable $R$-modules is equivalent to the category $\Mod_R$ of all $R$-modules. If $R$ is not coherent, then $\Mod_R^{\fp}$ is not abelian, so $\ca{A}$ being abelian is not a necessary condition for $\Ind(\ca{A})$ to be abelian. Our goal is to characterize categories whose category of ind-objects is abelian.
 
% Recall that an epimorphism in an $R$-linear category is called \emph{regular} if it is a cokernel.

\begin{dfn}\label{dfn:ind_abelian}
 A finitely cocomplete\footnote{A category is called \emph{finitely cocomplete} if it has all finite colimits. An $R$-linear category is finitely cocomplete if and only if it has finite direct sums and cokernels.} $R$-linear category $\ca{A}$ is called \emph{ind-abelian} if the following conditions are satisfied:
 \begin{enumerate}
  \item[(i)] Every epimorphism in $\ca{A}$ is \emph{regular}, that is, it is a cokernel of some other morphism;
  \item[(ii)]
   If the bottom row in the commutative diagram
   \[
   \xymatrix{B \ar@{..>>}[r]^{p^{\prime}} \ar@{..>}[d]_{f^{\prime}} & A \ar[d]^{f} \ar[rd]^{0} \\
             W \ar[r]^{p} & V \ar@{->>}[r]^{q} & U \ar[r] & 0}
   \]
   of solid arrows is right exact, then there exists an epimorphism $p^{\prime}$ and a morphism $f^{\prime}$ making the above diagram commutative.
 \end{enumerate}
\end{dfn}

\begin{example}\label{example:fg_modules}
 The category of finitely presentable modules of a ring $R$ is ind-abelian. Indeed, since the kernel of an epimorphism between finitely presentable modules is finitely generated, we can write it as a quotient of a finitely generated free module. This shows that every epimorphism is regular.

 To see part~(ii) of the definition, note that the assumption implies that $f$ factors through the image of $p$. If the pullback along this factorization is finitely presentable, we are done. If not, we can still find a finite set of elements in it whose image generates $A$. We can then take $B$ to be any finitely generated free module surjecting onto these elements. 
\end{example}

 In \S \ref{section:ind_abelian} we prove the following theorem.
 
\begin{thm}\label{thm:ind_abelian}
 A finitely cocomplete $R$-linear category $\ca{A}$ is ind-abelian if and only if $\Ind(\ca{A})$ is abelian. 
\end{thm}

 The author's interest in this question stems from the desire to give a characterization of categories of finitely presentable quasi-coherent sheaves of algebraic stacks. Before stating this theorem in detail we outline an application.

%This category is ind-abelian, so it is finitely cocomplete. Therefore it makes sense to talk about right exact functors with domain $\QCoh_{\fp}(X)$. However, left exactness cannot be expressed in the absence of kernels. The correct generalization of left exactness we need to consider is flatness.

 %This generalization allows us to recognize categories of coherent sheaves of Adams stacks (which are certain stacks on the $\fpqc$-site $\Aff$ of affine schemes). Before we explain this in detail we outline an application of this recognition theorem. 

%In \S \ref{section:ind_abelian}, we introduce the notion of an ind-abelian category, and we prove that $\ca{A}$ is ind-abelian if and only if $\Ind(\ca{A})$ is abelian.

\subsection{The Deligne tensor product of categories of coherent sheaves}\label{section:deligne_tensor_intro}
 Fix a commutative ring $R$. All the schemes we consider are schemes over $\Spec(R)$. Recall that a noetherian scheme has the \emph{resolution property} if every coherent sheaf is a quotient of a locally free sheaf of finite rank. 
 
 If $\ca{A}$ and $\ca{B}$ are $R$-linear abelian categories, their \emph{Deligne tensor product} $\ca{A}\boxtimes \ca{B}$, if it exists, is the universal abelian category with a functor
\[
 \ca{A} \times \ca{B} \rightarrow \ca{A} \boxtimes \ca{B}
\]
 which is $R$-linear and right exact in each variable (see \cite[\S 5]{DELIGNE}). One of the aims of this article is to prove the following theorem.

\begin{thm}\label{thm:scheme_product_to_coproduct}
 Let $X$ and $Y$ be semi-separated noetherian schemes with the resolution property. If $X\times Y$ is again noetherian, then there is an equivalence
\[
 \Coh(X\times Y) \simeq \Coh(X)\boxtimes \Coh(Y)
\]
 of symmetric monoidal categories.
\end{thm}

 This gives a partial answer to a question posed by Martin Brandenburg on MathOverflow (\url{http://mathoverflow.net/questions/55735}). The following provides a strategy for proving this theorem:
\begin{enumerate}
 \item[(i)]  We show that the contravariant functor $\Coh(-)$ which sends a noetherian scheme with the resolution property to its category of coherent sheaves gives an embedding of the category of such schemes in the 2-category of $R$-linear abelian tensor categories.
 \item[(ii)] We show that $X \times Y$ has the resolution property if $X$ and $Y$ do.
\item[(iii)] We show that the Deligne tensor product of two $R$-linear abelian tensor categories is a (bicategorical) coproduct in the 2-category of $R$-linear abelian tensor categories.
\item[(iv)] We show that the Deligne tensor product $\Coh(X) \boxtimes \Coh(Y)$ lies in the image of $\Coh(-)$, that is, it is equivalent as $R$-linear tensor category to $\Coh(Z)$ for some noetherian scheme $Z$ with the resolution property.
\end{enumerate}

 If we can prove (i)-(iv), then we can prove Theorem~\ref{thm:scheme_product_to_coproduct} as follows. From (i) and (ii) we deduce that $\Coh(X\times Y)$ is a (bicategorical) coproduct in the image of $\Coh(-)$. Note that there is a priori no reason why $\Coh(X\times Y)$ should be a coproduct among \emph{all} $R$-linear abelian tensor categories. But from (iii) and (iv) we know that $\Coh(X)\boxtimes \Coh(Y)$ is a coproduct in the image of $\Coh(-)$, so it must be equivalent to $\Coh(X\times Y)$.

 Unfortunately there are two major difficulties with this strategy. First, the Deligne tensor product of two $R$-linear abelian tensor categories need not exist. The second difficulty lies in the fourth step, where we have to show that a certain abstract abelian tensor category is the category of coherent sheaves on some noetherian scheme $Z$ with the resolution property.

 To address the first problem it is convenient to drop the noetherian assumption. If we do that, there is no reason to expect that the finitely presentable quasi-coherent sheaves on $X\times Y$ coincide with the coherent ones. If we want to keep fact (i), we have to consider the category $\QCoh_{\fp}(X)$ of finitely presentable sheaves instead of the category $\Coh(X)$ of coherent sheaves. This category is in general not abelian, but it has cokernels, hence it is finitely cocomplete. Instead of the Deligne tensor product we can consider the analogous universal property among all finitely cocomplete $R$-linear categories. Kelly has shown that this tensor product always exists \cite[\S 6.5]{KELLY_BASIC}. We call it the \emph{right exact} tensor product.

 To address the second difficulty of showing that the right exact tensor product
 $\Coh(X)\boxtimes \Coh(Y)$ is of the form $\Coh(Z)$ for some scheme $Z$, we would like to use a Tannakian recognition theorem. To the author's knowledge there is no recognition theorem that characterizes categories of coherent sheaves schemes directly. However, in \cite{SCHAEPPI_STACKS}, the author proved such a recognition theorem for a class of algebraic stacks which contains all the (discrete) stacks represented by semi-separated noetherian schemes with the resolution property (see \cite[Theorem~1.1.2]{SCHAEPPI_STACKS}). Thus, instead of showing that Theorem~\ref{thm:scheme_product_to_coproduct} holds for schemes, we will prove its generalization to a suitable class of algebraic stacks.

 Therefore, to solve both difficulties at once, we have to consider categories $\QCoh_{\fp}(X)$ of finitely presentable quasi-coherent sheaves on algebraic stacks $X$. These categories need not be abelian in general, so the recognition theorem from \cite{SCHAEPPI_STACKS} does not apply to them. But the category $\QCoh_{\fp}$ is ind-abelian. Our first aim is therefore to generalize the Tannakian recognition theorem from \cite{SCHAEPPI_STACKS} to the context of ind-abelian categories.

%the Tannakian recognition theorem for categories of finitely presentable coherent sheaves mentioned in \S \ref{section:ind_abelian_intro}. In \cite{SCHAEPPI_STACKS}, the author proved a recognition for categories of coherent sheaves of certain algebraic stacks.

%Thus, instead of proving Theorem~\ref{thm:scheme_product_to_coproduct} directly, we prove the more general theorem below of which it is a consequence. We recall the definition of Adams stacks in \S \ref{section:recognition_theorem_intro} below.

\subsection{The Tannakian recognition theorem for Adams stacks}\label{section:recognition_theorem_intro}
 Following Naumann \cite{NAUMANN}, we call a stack on the $\fpqc$-site $\Aff_R$ of affine schemes over $R$ \emph{algebraic} if it is  associated to a flat affine groupoid. An $\fpqc$-stack $X$ is algebraic if and only if it has an affine diagonal and there exists a faithfully flat morphism $X_0 \rightarrow X$ for some affine scheme $X_0$. An algebraic stack is \emph{coherent} if the category $\QCoh_{\fp}(X)$ of finitely presentable sheaves is abelian. In this case a sheaf is locally presentable if and only if it is coherent.

 An algebraic stack $X$ over $R$ has the \emph{strong resolution property} if the dual objects form a generator of $\QCoh(X)$, that is, for every $M \in \QCoh(X)$, there exists an epimorphism
\[
 \xymatrix{ \bigoplus_{i\in I} M_i \ar[r] & M }
\]
 of quasi-coherent sheaves where all the $M_i$ have duals. An algebraic stack with the strong resolution property is called an \emph{Adams stack}.

 Every quasi-compact and semi-separated scheme $X$ represents a discrete algebraic stack. It is an Adams stack if $X$ has the strong resolution property. In particular, every semi-separated noetherian scheme with the resolution property represents an Adams stack.

 Before stating the Tannakian recognition theorem for Adams stacks we recall the recognition theorem from \cite{SCHAEPPI_STACKS}. Let $\ca{A}$ be a symmetric monoidal abelian $R$-linear category, and let $B$ be a commutative $R$-algebra. Recall that
\[
 w \colon \ca{A} \rightarrow \Mod_B
\]
 is called a \emph{fiber functor} if it is faithful, exact, and strong symmetric monoidal. The category $\ca{A}$ is called \emph{weakly Tannakian} if there exists a fiber functor on $\ca{A}$ and if for every $A \in \ca{A}$ there exists a dualizable object $A^{\prime} \in \ca{A}$ and an epimorphism $A^{\prime} \rightarrow A$. The recognition theorem from \cite{SCHAEPPI_STACKS} states that $\ca{A}$ is weakly Tannakian if and only if there exists a coherent algebraic stack $X$ with the resolution property and an equivalence
\[
 \ca{A} \simeq \Coh(X)
\]
 of symmetric monoidal categories (see \cite[Theorem~1.1.2]{SCHAEPPI_STACKS}).

 The notion of ind-abelian categories allows us to generalize this recognition theorem to give a characterization of categories of the form $\QCoh_{\fp}(X)$ where $X$ is a not necessarily coherent Adams stack. 

\begin{dfn}\label{dfn:weakly_tannakian}
 Let $\ca{A}$ be an ind-abelian right exact symmetric monoidal $R$-linear category, and let $B$ be a commutative $R$-algebra. A functor
\[
 w \colon \ca{A} \rightarrow \Mod_B
\]
 is called a \emph{fiber functor} if it is faithful, flat, and right exact. 

 If $\ca{A}$ satisfies the conditions:
\begin{enumerate}
 \item[(i)] There exists a fiber functor $w \colon \ca{A} \rightarrow \Mod_B$ for some commutative $R$-algebra $B$;
\item[(ii)] For all objects $A \in \ca{A}$ there exists an epimorphism $A^{\prime} \rightarrow A$ such that $A^{\prime}$ has a dual;
\end{enumerate}
 it is called \emph{weakly Tannakian}.
\end{dfn}

 See \S \ref{section:flat_functors} for a definition of flat functors. In \S \ref{section:weakly_tannakian} we prove the following generalization of the recognition theorem from \cite{SCHAEPPI_STACKS}.

\begin{thm}\label{thm:recognition}
 Let $R$ be a commutative ring. An $R$-linear ind-abelian category $\ca{A}$ is weakly Tannakian if and only if there exists an Adams stack $X$ over $R$ and an equivalence
\[
 \ca{A} \simeq \QCoh_{\fp}(X)
\]
 of symmetric monoidal $R$-linear categories.
\end{thm}

 Using this we can prove the following theorem, of which Theorem~\ref{thm:scheme_product_to_coproduct} is a consequence.

\begin{thm}\label{thm:stack_product_to_coproduct}
 Let $X$ and $Y$ be Adams stacks over $R$. Then there is an equivalence
\[
 \QCoh_{\fp}(X \times Y) \simeq  \QCoh_{\fp}(X) \boxtimes  \QCoh_{\fp}(Y)
\]
 of symmetric monoidal categories, where $\boxtimes$ denotes Kelly's tensor product of finitely cocomplete $R$-linear categories.
\end{thm}

 The proof follows the strategy outlined in \S \ref{section:deligne_tensor_intro}. The pseudofunctor which sends an Adams stack $X$ to $\QCoh_{\fp}(X)$ is an embedding by \cite[Theorem~1.3.3]{SCHAEPPI_STACKS}. Thus step~(i) of the strategy works for Adams stacks. The analogue of step~(ii) for Adams stacks is a consequence of the following theorem, which we prove in \S \ref{section:adams}.

\begin{thm}\label{thm:adams_fiber_product}
 The class of Adams stacks is closed under fiber products.
\end{thm}

 Note that this theorem also gives new examples of schemes with the resolution property. In \cite[Theorem~5.2]{GROSS}, Gross has shown that all separated surfaces over an excellent base ring have the resolution property. In general, the classical criteria for the existence of locally free resolutions do not apply to finite products of such surfaces.

 To prove part~(iii), we have to show that the Kelly tensor product of two right exact symmetric monoidal categories is a bicategorical coproduct. This is a consequence of the purely categorical fact that tensor products of symmetric pseudomonoids are bicategorical coproducts (see Theorem~\ref{thm:pseudomonoid_coproduct}), which is a categorification of the fact that tensor products of commutative algebras are coproducts.

 In \S \ref{section:kelly_tensor} we conclude the proof of Theorem~\ref{thm:stack_product_to_coproduct} by showing that weakly Tannakian ind-abelian categories are closed under Kelly's tensor product.

 Throughout the paper we fix a commutative ring $R$. Most of the categories we consider are enriched in the category of $R$-modules. We frequently use basic concepts from the theory of enriched categories such as left Kan extensions and dense functors. These are discussed in detail in \cite[\S\S 4-5]{KELLY_BASIC}. A brief overview can also be found in \cite[\S 2]{SCHAEPPI_STACKS}. We recall the definition of Kelly's tensor product of finitely cocomplete categories in \S \ref{section:kelly_tensor_field}. This is a special case of the tensor product described in \cite[\S 6.5]{KELLY_BASIC}.

%(part iv) of the strategy).

% In \S \ref{thm:kelly_tensor} we show how this theorem can be used to complete the proof of Theorem~\ref{thm:stack_product_to_coproduct}.

\section*{Acknowledgments}
 This paper was written at the University of Chicago as part of my graduate studies under the supervision of Prof.\ Peter May. I thank Peter May for our weekly discussions and for his help with the organization of this paper. 

 The contents of \S \ref{section:adams} were worked out at Macquarie University during a stay made possible by an Endeavour Research Fellowship from the Australian government. I thank Richard Garner, Steve Lack, and Ross Street for various helpful discussions during that time. I am particularly indebted to Richard Garner for his suggestion to use the regular Grothendieck topology to prove that ind-objects form an abelian category.

 %This class includes all quasi-compact and quasi-separated schemes with the resolution property

% It is a well-known fact that the category $\Ind(\ca{A})$ of ind-objects of an abelian category $\ca{A}$ is again abelian. On the other hand, for any commutative ring $R$, the category of ind-objects of the category $\Mod_R^{\fp}$ finitely presentable $R$-modules is the category $\Mod_R$ of all $R$-modules. If $R$ is not coherent, then $\Mod_R^{\fp}$ is not abelian, so $\ca{A}$ being abelian is not a necessary condition for $\Ind(\ca{A})$ to be abelian. In \S \ref{section:ind_abelian}, we introduce the notion of an ind-abelian category, and we prove that $\ca{A}$ is ind-abelian if and only if $\Ind(\ca{A})$ is abelian.

\section{Ind-abelian categories}\label{section:ind_abelian}

\subsection{Proof of Theorem~\ref{thm:ind_abelian}}
 We first show that $\ca{A}$ is ind-abelian if $\Ind(\ca{A})$ is abelian. Example~\ref{example:fg_modules} is a special case of this result, and the proof of the general fact follows the same outline. We first recall some basic facts about locally finitely presentable abelian categories.
 
 An object $C$ in a category $\ca{C}$ is called \emph{finitely presentable} if $\ca{C}(C,-)$ preserves filtered colimits. This means that any morphism from $C$ to a filtered colimit $\colim^{i \in \ca{I}} A_i$ factors through one of the structural morphisms $A_i \rightarrow \colim^{i \in \ca{I}} A_i$, and that for any two such factorizations
\[
 \xymatrix{C \ar[r] & A_i} \quad \text{and} \quad  \xymatrix{C \ar[r] & A_j}
\]
 there exists an object $k \in \ca{I}$ with morphisms $i \rightarrow k$ and $j \rightarrow k$ such that the diagram
\[
 \xymatrix{C \ar[r] \ar[d] & A_i \ar[d] \\ A_j \ar[r] & A_k}
\]
 commutes.

 The object $C$ is called \emph{finitely generated} if $\ca{C}(C,-)$ preserves directed colimits of monomorphisms. A cocomplete category $\ca{C}$ is called \emph{locally finitely presentable} if every object is a filtered colimit of finitely presentable objects. In this case we write $\ca{C}_{\fp}$ for the full subcategory of finitely presentable objects. Locally finitely presentable categories were introduced by Gabriel and Ulmer \cite{GABRIEL_ULMER}.

 The category $\ca{C}$ is called \emph{locally finitely generated} if every object is a directed union of finitely generated objects. If $\ca{C}$ is a locally finitely presentable category, then the finitely generated objects are precisely the regular quotients of the finitely presentable objects. If $\ca{C}$ is also abelian we can therefore use image factorizations to show that $\ca{C}$ is also locally finitely generated. We give a proof of the following standard lemma to provide an example for how these definitions are used.
 
 \begin{lemma}\label{lemma:finitely_generated}
 Let $\ca{C}$ be a locally finitely presentable abelian category. If the objects $B$ and $C$ in the exact sequence
 \[
 \xymatrix{0 \ar[r] & A \ar[r] & B \ar[r] & C \ar[r] & 0}
 \]
 are finitely presentable, then $A$ is finitely generated. 
\end{lemma}
 
\begin{proof}
 Write $A$ as a union of finitely generated subobjects $A_i$. Then $C$ is the directed colimit of the cokernels $B \slash A_i$. Since $B$ is finitely presentable, there exists a factorization
\[
 \xymatrix{ C \ar[r]^-{s} & B \slash A_i \ar[r] & C }
\]
 of the identity on $C$. By construction the two composites
\[
 \xymatrix{A \ar[r] & C \ar[r]^-{s} & B \slash A_i \ar[r] & C }
\quad \text{and} \quad
\xymatrix{A \ar[r] & B \slash A_i \ar[r] & C}
\]
 are equal. Local presentability of $A$ implies that there exists an index $j$ such that the diagram
\[
 \xymatrix{A \ar[rr] \ar[d] && B \slash A_i \ar[d] \\ C \ar[r]^-{s} & B\slash A_i  \ar[r] & B \slash A_j}
\]
 commutes. From this we deduce that $A$ is a subobject of $A_j$, hence that $A=A_j$.
\end{proof}

\begin{prop}\label{prop:fp_objects_ind_abelian}
 Let $\ca{C}$ be a finitely presentable abelian $R$-linear category. Then $\ca{C}_{\fp}$ is ind-abelian.
\end{prop}
 
\begin{proof}
 Let $p \colon B \rightarrow C$ be an epimorphism in $\ca{C}_{\fp}$, and let $K \in \ca{C}$ be the kernel of $p$. From Lemma~\ref{lemma:finitely_generated} it follows that $K$ is finitely generated. Thus there exists an epimorphism $A \rightarrow K$ where $A \in \ca{C}_{\fp}$. It follows that the sequence
  \[
 \xymatrix{ A \ar[r] & B \ar[r]^{p} & C \ar[r] & 0}    
  \]
  in $\ca{C}_{\fp}$ is right exact, hence that $p$ is a regular epimorphism.
  
 Now consider a solid arrow diagram in $\ca{C}_{\fp}$ as in part (ii) of Definition~\ref{dfn:ind_abelian}. Let $k \colon K \rightarrow V $ be the kernel of $q$ in $\ca{C}$. Then there exists a unique morphism $g \colon A \rightarrow K$ such that $kg=f$. Let
  \[
  \xymatrix{P \ar@{->>}[r]^{p^{\prime}} \ar[d] & A \ar[d]^g \\ W \ar@{->>}[r] & K}
  \]
 be a pullback square in $\ca{C}$. In general, $P$ itself need not be finitely generated, but we can always find a finitely generated subobject $P_0$ of $P$ such that the composite
  \[
  \xymatrix{P_0 \ar[r] & P \ar@{->>}[r]^-{p^{\prime}} & A}
  \]
 is still an epimorphism. Indeed, this follows from the fact that $P$ is the directed union of its finitely generated subobjects and the existence of image factorizations. We obtain the desired square of dotted arrows in condition~(ii) of Definition~\ref{dfn:ind_abelian} by choosing an epimorphism $B \rightarrow P_0$ such that $B$ is finitely presentable.
\end{proof}

 By applying the above proposition to the case $\ca{C}=\Ind(\ca{A})$ we can already prove one half of Theorem~\ref{thm:ind_abelian}, namely that $\ca{A}$ is ind-abelian if $\Ind(\ca{A})$ is abelian. Showing the converse is a bit more involved. We use the identification
 \[
  \Ind(\ca{A}) \simeq \Lex[\ca{A}^{\op},\Mod_R]
 \]
 between the category of ind-objects of a finitely cocomplete $R$-linear category $\ca{A}$ and the category of left exact $R$-linear functors $\ca{A}^{\op} \rightarrow \Mod_R$. The latter sits naturally in the category of all $R$-linear presheaves on $\ca{A}$. Our strategy is to show that every ind-abelian category has a Grothendieck topology whose sheaves are precisely the left exact presheaves. Since the associated sheaf functor is left exact, it will follow that $\Lex[\ca{A}^{\op},\Mod_R]$ is abelian. We present the argument in slightly greater generality, which we will use in Section~\ref{section:kelly_tensor_field} to show that ind-abelian categories over fields are closed under the Kelly tensor product.

\begin{dfn}\label{dfn:inductive_class}
 Let $\ca{A}$ be a finitely cocomplete $R$-linear category. If $\Sigma$ is a class of right exact sequences 
\[
 \xymatrix{W \ar[r]^{p} & V \ar[r]^{q} & U \ar[r] & 0}
 \]
 in $\ca{A}$, we write $R(\Sigma)$ for the class of all the morphisms $q$ which appear in an exact sequence in $\Sigma$ as above.
 
 The class $\Sigma$ is called an \emph{ind-class} if it has the following two properties:
 \begin{enumerate}
 \item[(i)]
 For all $q \in R(\Sigma)$, there is a right exact sequence
 \[
 \xymatrix{V \ar[r]^{p} & U \ar[r] & Z \ar[r]^{\cong} & 0} 
 \]
 which lies in $\Sigma$.
 \item[(ii)] 
 For all exact sequences
 \[
 \xymatrix{W \ar[r]^{p} & V \ar[r]^{q} & U \ar[r] & 0}
 \]
 in $\Sigma$ and all morphisms $f \colon A \rightarrow V$ with $qf=0$, there exists a morphism $p^{\prime} \in R(\Sigma)$ and a morphism $f^{\prime} \colon B \rightarrow W$ in $\ca{A}$ such that the diagram
 \[
 \xymatrix{ B \ar@{->>}[r]^{p^{\prime}} \ar[d]_{f^{\prime}} & A \ar[d]^{f} \ar[rd]^{0} \\ W \ar[r]_{p} & V \ar@{->>}[r]_{q} & U}
 \]
 is commutative.
 \end{enumerate}
\end{dfn}

\begin{example}
 If $\ca{A}$ is an ind-abelian category, then the class of \emph{all} right exact sequences is an ind-class. For this class, $R(\Sigma)$ consists of the epimorphisms in $\ca{A}$.
\end{example}

 A \emph{singleton coverage} (or \emph{singleton Grothendieck pretopology}) on an $R$-linear category $\ca{A}$ is a class of morphisms $\Xi$ such that for all solid arrow diagrams
 \[
 \xymatrix{V^{\prime} \ar@{..>}[r]^{q^{\prime}} \ar@{..>}[d]_{f^{\prime}} & U^{\prime} \ar[d]^{f} \\ V \ar[r]_{q} & U}
 \]
 with $q \in \Xi$, there exist dotted arrows $q^{\prime} \in \Xi$ and $f^{\prime}$ such that the above diagram commutes. An $R$-linear functor $F \colon \ca{A}^{\op} \rightarrow \Mod_R$ is called a \emph{sheaf} for the coverage $\Xi$ if the following two conditions are satisfied:
 \begin{enumerate}
 \item[(i)] The functor $F$ sends morphisms in $\Xi$ to monomorphisms;
 \item[(ii)]  For all $q \colon V \rightarrow U \in \Xi$, an element $x \in F(V)$ lies in the image of $F(q)$ if and only if $F(f)(x)=0$ for all $f \colon A \rightarrow V$ with $qf=0$.
 \end{enumerate}

\begin{lemma}\label{lemma:inductive_coverage}
 Let $\Sigma$ be an ind-class in the finitely cocomplete $R$-linear category $\ca{A}$. Then $R(\Sigma)$ is a singleton coverage.
\end{lemma}

\begin{proof}
 The defining condition for a coverage is obtained by applying condition~(ii) of Definition~\ref{dfn:inductive_class} to the exact sequence of condition~(i).
\end{proof}

\begin{prop}\label{prop:sheaf_iff_exact}
 Let $\ca{A}$ be a finitely cocomplete $R$-linear category and let $\Sigma$ be an ind-class in $\ca{A}$. Then $F \colon \ca{A}^{\op} \rightarrow \Mod_R$ is a sheaf for the coverage $R(\Sigma)$ if and only if $F$ preserves all the exact sequences in $\Sigma$.
\end{prop}

 Before proving this proposition we show how it implies that $\Ind(\ca{A})$ is abelian if $\ca{A}$ is ind-abelian.
 
\begin{proof}[Proof of Theorem~\ref{thm:ind_abelian}]
 Let $\ca{A}$ be a finitely cocomplete $R$-linear category. First suppose that $\Ind(\ca{A})$ is abelian. Since $\Ind(\ca{A})$ is always locally finitely presentable, with $\Ind(\ca{A})_{\fp}=\ca{A}$, Proposition~\ref{prop:fp_objects_ind_abelian} implies that $\ca{A}$ is ind-abelian.
 
 Conversely, suppose that $\ca{A}$ is ind-abelian. The class $\Sigma$ of \emph{all} right exact sequences is an ind-class by definition of ind-abelian categories (see Definitions~\ref{dfn:ind_abelian} and \ref{dfn:inductive_class}). By Proposition~\ref{prop:sheaf_iff_exact}, left exact presheaves are precisely the sheaves for the coverage $R(\Sigma)$ of (regular) epimorphisms. Since the associated sheaf functor is left exact it follows that $\Ind(\ca{A}) \simeq \Lex[\ca{A}^{\op},\Mod_R]$ is a lex-reflective subcategory of the abelian category $[\ca{A}^{\op},\Mod_R]$ of $R$-linear presheaves. It follows that $\Ind(\ca{A})$ is abelian.
\end{proof}

 To prove Proposition~\ref{prop:sheaf_iff_exact} we use the following lemmas.
 
 \begin{lemma}\label{lemma:exact_implies_sheaf}
 If $F$ preserves exact sequences in $\Sigma$, then $F$ is a sheaf for the singleton coverage $R(\Sigma)$. In particular, representable functors on $\ca{A}$ are sheaves for $R(\Sigma)$.
 \end{lemma}
 
 \begin{proof}
  Since $F(q)$ is the kernel of $F(p)$, it is a monomorphism, so it only remains to show that the second sheaf condition holds. The image of $F(q)$ coincides with the kernel of $F(p)$ by assumption, so it suffices to show that $F(p)(x)=0$ for $x \in F(V)$ if and only if $F(f)(x)=0$ for all $f \colon A \rightarrow V$ with $qf=0$. Since $qp=0$, one of these implications is obvious. To see the other, assume that $F(p)(x)=0$ and let $f \colon A \rightarrow V$ be such that $qf=0$. From the definition of ind-classes we find that there exist morphisms $p^{\prime} \in R(\Sigma)$ and $f^{\prime}$ such that the diagram
  \[
  \xymatrix{ B \ar@{->>}[r]^{p^{\prime}} \ar@{->}[d]_{f^{\prime}} & A \ar[d]^{f} \ar[rd]^{0} \\ W \ar[r]_{p} & V \ar@{->>}[r]_{q} & U} 
  \]
 commutes. Thus
 \[
 F(p^{\prime}) \circ F(f)(x)=F(f^{\prime})\circ F(p)(x)=0 \smash{\rlap{.}} 
 \]
 Since $p^{\prime}$ lies in $R(\Sigma)$ it follows that $F(p^{\prime})$ is a monomorphism, which in turn implies that $F(f)(x)=0$.
 \end{proof}

 \begin{lemma}\label{lemma:yoneda_exact}
 In the situation of Proposition~\ref{prop:sheaf_iff_exact}, let $\Sh(\ca{A})$ denote the category of sheaves for the singleton coverage $R(\Sigma)$, and let
 \[
  \xymatrix{W \ar[r]^{p} & V \ar[r]^{q} & U \ar[r] & 0}
 \]
 be an exact sequence in $\Sigma$. Then the sequence
 \[
  \xymatrix{\ca{A}(-,W) \ar[r]^-{\ca{A}(-,p)} & \ca{A}(-,V) \ar[r]^-{\ca{A}(-,q)} & \ca{A}(-,U) \ar[r] & 0} 
 \]
 in $\Sh(\ca{A})$ is exact.
 \end{lemma}

\begin{proof}
 From the Yoneda lemma
 \[
\xymatrix{ \Sh(\ca{A}) \bigl(\ca{A}(-,A), F\bigr)  \ar[d]_{\cong} \ar[rrr]^{ \Sh(\ca{A})(\ca{A}(-,q), F) } &&& \Sh(\ca{A})\bigl(\ca{A}(-,V), F\bigr) \ar[d]^{\cong} \\ 
 FU \ar[rrr]^{F(q)} &&& FV }
 \]
 and the first sheaf condition it follows that $\ca{A}(-,q)$ is an epimorphism in the category $\Sh(\ca{A})$ for all $q$ in $R(\Sigma)$. It remains to show that $\ca{A}(-,p)$ induces an epimorphism onto the kernel $K$ of $\ca{A}(-,q)$. Let $h \colon K \rightarrow F$ be a morphism in $\Sh(\ca{A})$ such that the composite
 \[
  \xymatrix{ \ca{A}(-,W) \ar[r] & K \ar[r]^h & F}
 \]
 is zero. We have to show that $h=0$.
 
 Since the morphisms $\ca{A}(-,A) \rightarrow K$ where $A$ runs through all the objects of $\ca{A}$ are jointly epimorphic, it suffices to show that their composite with $h$ is always zero. By Yoneda there exists a morphism $f \colon A \rightarrow V$ such that the upper triangle in the solid arrow diagram
 \[
 \xymatrix{ \ca{A}(-,B) \ar@{..>}[d]_{\ca{A}(-,f^{\prime})} \ar@{..>>}[r]^{\ca{A}(-,p^{\prime})} & \ca{A}(-,A) \ar[d] \ar[rd]^{\ca{A}(-,f)} \\
  \ca{A}(-,W) \ar[rd]_{0} \ar[r] & K \ar[d]^{h} \ar[r] & \ca{A}(-,V) \ar[r]^{\ca{A}(-,q)} & \ca{A}(-,U) \\ & F}
 \]
 commutes. Since $\ca{A}(-,f)$ factors through the kernel $K$ of $\ca{A}(-,q)$, we know that $qf=0$. From the definition of ind-classes (see Definition~\ref{dfn:inductive_class}) it follows that we can choose $p^{\prime} \in R(\Sigma)$ and $f^{\prime}$ in $\ca{A}$ in place of the dotted arrows such that the above diagram is commutative. We have already observed that $\ca{A}(-,p^{\prime})$ is an epimorphism in the category $\Sh(\ca{A})$. Thus the composite $\ca{A}(-,A) \rightarrow F$ is indeed zero, as claimed.
\end{proof}

\begin{proof}[Proof of Proposition~\ref{prop:sheaf_iff_exact}]
Lemma~\ref{lemma:exact_implies_sheaf} implies that any $F \colon \ca{A}^{\op} \rightarrow \Mod_R $ which preserves exact sequences in $\Sigma$ is a sheaf for the singleton coverage $R(\Sigma)$. It remains to show the converse.

 Let $F$ be a sheaf for $R(\Sigma)$, and let
  \[
  \xymatrix{W \ar[r]^{p} & V \ar[r]^{q} & U \ar[r] & 0}
 \]
 be an exact sequence in $\Sigma$. By Lemma~\ref{lemma:yoneda_exact}, the sequence
 \[
  \xymatrix{\ca{A}(-,W) \ar[r]^-{\ca{A}(-,p)} & \ca{A}(-,V) \ar[r]^-{\ca{A}(-,q)} & \ca{A}(-,U) \ar[r] & 0} 
 \]
 in $\Sh(\ca{A})$ is exact as well. The contravariant functor $\Sh(\ca{A})(-,F)$ sends right exact sequences to left exact sequences. Thus the Yoneda lemma
 \[
  \xymatrix{0 \ar[d] \ar@{=}[r] & 0 \ar[d] \\
  \Sh(\ca{A})\bigl(\ca{A}(-,U),F\bigr) 
   \ar[d]_-{\Sh(\ca{A})(\ca{A}(-,q),F)} \ar[r]^-{\cong} & F(U) \ar[d]^{F(q)} \\ 
  \Sh(\ca{A})\bigl(\ca{A}(-,V),F\bigr)
   \ar[d]_-{\Sh(\ca{A})(\ca{A}(-,p),F)} \ar[r]^-{\cong} & F(V) \ar[d]^{F(p)} \\  
  \Sh(\ca{A})\bigl(\ca{A}(-,U),F\bigr) \ar[r]^-{\cong} & F(W) }  
 \]
 implies that $F$ preserves exact sequences in $\Sigma$.
\end{proof}

\subsection{Closing subcategories of \texorpdfstring{$R$}{R}-linear categories under colimits}
 Proposition~\ref{prop:sheaf_iff_exact} allows us to give a more explicit description of the Kelly tensor product of two ind-abelian categories over a field. In order to do that we need to recall the procedure of closing a subcategory under a class of colimits (see \cite[\S 3.5]{KELLY_BASIC}).
 
 Given a cocomplete $R$-linear category $\ca{C}$ and a full subcategory $\ca{A}$, the \emph{closure} of $\ca{A}$ under finite colimits is the union of the countable sequence
 \[
  \ca{A}=\ca{C}_0 \subseteq \ca{C}_{1} \subseteq \ca{C}_{2} \subseteq \ldots \smash{\rlap{,}}
 \]
 where $\ca{C}_{i+1}$ is the full subcategory of $\ca{C}$ whose objects are colimits of finite diagrams in $\ca{C}_{i}$. In general, this sequence need not stabilize after a finite number of steps, hence the resulting closure is often very difficult to describe explicitly. The following proposition provides a class of examples where the sequence terminates after two steps.
 
 \begin{prop}\label{prop:fp_in_lex_sheaves}
 Let $\Sigma$ be an ind-class, and let $\Lex_{\Sigma} [\ca{A}^{\op},\Mod_R]$ be the full subcategory of $R$-linear presheaves on $\ca{A}$ which send all the right exact sequences in $\Sigma$ to left exact sequences. Then the following conditions hold:
 \begin{enumerate}
 \item[(i)] The category $\Lex_{\Sigma} [\ca{A}^{\op},\Mod_R]$ is locally finitely presentable and abelian;
 \item[(ii)] For every finitely presentable object $F$ in $\Lex_{\Sigma} [\ca{A}^{\op},\Mod_R]$, there exists a right exact sequence
\[
 \bigoplus_{j=1}^m \ca{A}(-,B_j) \longrightarrow \bigoplus_{i=1}^n \ca{A}(-,A_i) \longrightarrow F \rightarrow 0 \smash{\rlap{;}}
\]
 \item[(iii)] The closure of $\ca{A}$ in $\Lex_{\Sigma} [\ca{A}^{\op},\Mod_R]$ under finite colimits coincides with the full subcategory of finitely presentable objects. It is in particular ind-abelian.
 \end{enumerate}
 \end{prop}
 
\begin{proof}
 Filtered colimits in $\Lex_{\Sigma} [\ca{A}^{\op},\Mod_R]$ are computed as in $[\ca{A}^{\op},\Mod_R]$ because finite limits commute with filtered colimits. It follows that the representables form a dense generator of finitely presentable objects in $\Lex_{\Sigma} [\ca{A}^{\op},\Mod_R]$. Thus $\Lex_{\Sigma} [\ca{A}^{\op},\Mod_R]$ is locally finitely presentable. By Proposition~\ref{prop:sheaf_iff_exact}, 
 \[
 \Lex_{\Sigma}[\ca{A}^{\op},\Mod_R]=\Sh_{R(\Sigma)}(\ca{A}) \smash{\rlap{,}}
 \]
so it is lex-reflexive and therefore abelian. This proves (i).
 
 For every finitely presentable object $F$ there exists an epimorphism
 \[
 \bigoplus_{i=1}^n \ca{A}(-,A_i) \longrightarrow F
 \] 
 since the representable functors form a dense generator. Its kernel is finitely generated by Lemma~\ref{lemma:finitely_generated}. Thus it admits an epimorphism from a finitely presentable object. Applying the above argument again we get a right exact sequence
 \[
 \bigoplus_{j=1}^m \ca{A}(-,B_j) \longrightarrow \bigoplus_{i=1}^n \ca{A}(-,A_i) \longrightarrow F \rightarrow 0 \smash{\rlap{,}}
 \]
 which shows that (ii) holds.
 
 It remains to show (iii). By (ii), every finitely presentable object lies in the closure of the representable functors under finite colimits. Conversely, we already observed that every representable functor is finitely presentable, showing that $\ca{A}$ is contained in $\Lex_{\Sigma}[\ca{A}^{\op},\Mod_R]_{\fp}$. The conclusion follows since finitely presentable objects are closed under finite colimits.
\end{proof}

\subsection{Kelly's tensor product of finitely cocomplete categories}\label{section:kelly_tensor_field}
 Given two $R$-linear categories $\ca{A}$ and $\ca{B}$, their \emph{tensor product} $\ca{A} \otimes \ca{B}$ has objects the pairs $(A,B)$ with $A \in \ca{A}$ and $B\in \ca{B}$, and the $R$-module of homomorphisms from $(A,B)$ to $(A^{\prime},B^{\prime})$ is given by $\ca{A}(A,A^{\prime})\otimes \ca{B}(B,B^{\prime})$ (where the tensor product is taken over the base ring $R$). To give an $R$-linear functor $\ca{A}\otimes \ca{B} \rightarrow \ca{C}$ amounts to giving two families of $R$-linear functors $(\ca{A} \rightarrow \ca{C})_{B \in \ca{B}}$ and $(\ca{B}\rightarrow \ca{C})_{A \in \ca{A}}$, subject to certain natural compatibility conditions (see \cite[Diagram~(1.21)]{KELLY_BASIC} for details). Thus $\ca{A}\otimes \ca{B}$ is universal among ``bilinear functors.'' The usual adjunction formula for tensor products is also valid, that is, there is an isomorphism
 \[
 [\ca{A}\otimes \ca{B},\ca{C}] \cong  [\ca{A},[\ca{B},\ca{C}]]
 \]
 of $R$-linear categories (see \cite[\S 2.3]{KELLY_BASIC}).
 
 Even if $\ca{A}$ and $\ca{B}$ are finitely cocomplete, there is no reason why $\ca{A}\otimes \ca{B}$ would be cocomplete as well. However, Kelly has shown that there exists a finitely cocomplete $R$-linear category $\ca{A}\boxtimes \ca{B}$ with a similar universal property in the world of finitely cocomplete $R$-linear categories and right exact $R$-linear functors. Writing $\Rex[\ca{A},\ca{B}]$ for the category of right exact $R$-linear functors, we have
 \[
 \Rex[\ca{A} \boxtimes \ca{B}, \ca{C}] \simeq \Rex[\ca{A},\Rex[\ca{B},\ca{C}]]
 \]
 for all finitely complete $R$-linear categories $\ca{A}$, $\ca{B}$, $\ca{C}$ (see \cite[Formula~(6.24)]{KELLY_BASIC}). We call this tensor product the \emph{right exact tensor product}. Note that the universal property of $\ca{A} \boxtimes \ca{B}$ is weaker than the one of $\ca{A} \otimes \ca{B}$. The former is universal only up to equivalence, whereas the latter is universal up to isomorphism.

 In order to prove the results of \S \ref{section:kelly_tensor}, we use the construction of $\ca{A}\boxtimes \ca{B}$ from \cite[\S 6.5]{KELLY_BASIC}, which we now recall. As an aside, this construction also allows us to prove that the class of ind-abelian categories over a field is closed under the right exact tensor product.

 Let $\ca{A}$ and $\ca{B}$ be two $R$-linear categories. Let $\Sigma$ be the class of sequences of the form
\[
  \xymatrix{(A,B) \ar[r]^{p\otimes \id_B} & (A^{\prime},B) \ar[r]^{q\otimes \id_B} & (A^{\prime\prime},B) \ar[r] & 0} 
\]
 in $\ca{A} \otimes \ca{B}$, where 
 \[
 \xymatrix{A \ar[r]^{p} & A^{\prime} \ar[r]^{q} & A^{\prime\prime} \ar[r] & 0}
 \]
 is a right exact sequence in $\ca{A}$, and similarly for right exact sequences in $\ca{B}$. These sequences are exact if $R$ is a field, but they need not be exact in general. We write
\[
 \Lex_\Sigma[(\ca{A}\otimes \ca{B})^{\op},\Mod_R]
\]
 for the full subcategory of presheaves which send all the sequences in $\Sigma$ to left exact sequences. Note that if $R$ is a field, then the representable functors on $\ca{A} \otimes \ca{B}$ are contained in $\Lex_\Sigma[(\ca{A}\otimes \ca{B})^{\op},\Mod_R]$, but this does not hold for general commutative rings $R$. However, the category $ \Lex_\Sigma[(\ca{A}\otimes \ca{B})^{\op},\Mod_R]$ is always a reflective subcategory of the presheaf category $[(\ca{A}\otimes \ca{B})^{\op},\Mod_R]$ (see \cite[Theorem~6.11]{KELLY_BASIC}).

\begin{prop}\label{prop:rex_tensor}
 The right exact tensor product $\ca{A} \boxtimes \ca{B}$ of two $R$-linear categories is given by the closure of the reflections of the representable functors in
\[
 \Lex_\Sigma[(\ca{A}\otimes \ca{B})^{\op},\Mod_R]
\]
 under finite colimits. The univeral bilinear right exact functor
\[
 Z \colon \ca{A} \otimes \ca{B} \rightarrow \ca{A} \boxtimes \ca{B}
\]
 is given by the composite
\[
 \xymatrix{ \ca{A} \otimes \ca{B} \ar[r]^-{Y} & [(\ca{A} \otimes \ca{B})^{\op},\Mod_R] \ar[r]^-{R} & \Lex_{\Sigma} [(\ca{A} \otimes \ca{B})^{\op},\Mod_R] }
\]
 of the Yoneda embedding and the reflection $R$.
\end{prop}

\begin{proof}
 This is \cite[Proposition~6.21]{KELLY_BASIC}, where $\Phi\mbox{-}\mathbf{Alg}=\Lex_{\Sigma}[(\ca{A}\otimes \ca{B})^{\op},\Mod_R]$ and $\ca{D}=\ca{A}\boxtimes \ca{B}$.
\end{proof}

 Thus, for $R=k$ a field, $\ca{A} \boxtimes \ca{B}$ is simply the closure of the representable functors under finite colimits in $\Lex_\Sigma[(\ca{A}\otimes \ca{B})^{\op},\Mod_R]$. We can therefore use Proposition~\ref{prop:fp_in_lex_sheaves} to show that for $R=k$ a field, ind-abelian categories are closed under the right exact tensor product.

\begin{prop}
 Let $\ca{A}$ and $\ca{B}$ be ind-abelian $k$-linear categories. Then $\ca{A} \boxtimes \ca{B}$ is ind-abelian.
\end{prop}

\begin{proof}
 Since we are working over a field, the class $\Sigma$ defined above consists of right exact sequences. By Proposition~\ref{prop:fp_in_lex_sheaves} it therefore suffices to show that $\Sigma$ is an ind-class. Condition~(i) follows from the fact that the sequence
\[
 \xymatrix{(V,B) \ar[r]^{q\otimes \id_B} & (U,B) \ar[r]^{0\otimes \id_A} & (0,B) \ar[r] & 0 }
\]
 lies in $\Sigma$ for all epimorphisms $q \colon V \rightarrow U$ in $\ca{A}$. To check that (ii) holds, let
\[
 \xymatrix{(W,B) \ar[r]^{p\otimes \id_B} & (V,B) \ar[r]^{q\otimes \id_B} & (U,B) \ar[r] & 0} 
\]
 be a sequence in $\Sigma$, and let $f \colon (A,B^{\prime}) \rightarrow (V,B)$ be a morphism with $q\otimes \id_B \circ f=0$. Choose a basis $(e_i)_{i \in I}$ of the vector space $\ca{B}(B^{\prime},B)$. Then there exist unique morphisms $f_i \colon A \rightarrow U$ such that $f=\sum_{i \in I} f_i \otimes e_i$. All but finitely many of the $f_i$ are zero, so without loss of generality we can write
\[
 f=\sum_{i=1}^n f_i \otimes e_i \smash{\rlap{.}}
\]
 From the fact that $\sum_{i \in I} qf_i \otimes e_i=0$ we deduce that $qf_i=0$ for all $i \in I$. Since $\ca{A}$ is ind-abelian, we can successively choose epimorphisms $p_i$ and morphisms $h_i$ such that the diagrams
\[
 \xymatrix{A_i \ar[d]_{h_i} \ar@{->>}[r]^{p_i} & A \ar[d]^{f_i \circ p_1 \circ \ldots \circ p_{i-1}} \\ W \ar[r]^-{p} & V }
\]
 commute. Then let $k_i=h_i \circ p_{i+1} \circ \ldots \circ p_n$ for $i<n$, $k_n=h_n$, and $p^{\prime}=p_1 \circ \ldots \circ p_n$. By construction, the diagram
\[
 \xymatrix{(A_n,B^{\prime}) \ar[d]_{\sum k_i \otimes e_i} \ar@{->>}[r]^-{p^{\prime} \otimes \id_{B^{\prime}}} & (A,B^{\prime}) \ar[d]^{\sum f_i \otimes e_i} \\ (W,B) \ar[r]^-{p \otimes \id_B} & (V,B) }
\]
 in $\ca{A}\otimes \ca{B}$ commutes. This shows that condition (ii) holds for this particular sequence in $\Sigma$. The case of a sequence in $\Sigma$ induced from a right exact sequence in $\ca{B}$ is proved analogously. The claim now follows from Proposition~\ref{prop:fp_in_lex_sheaves}~(iii).
\end{proof}

\section{Weakly Tannakian categories}\label{section:weakly_tannakian}

 \subsection{Recollections about flat functors}\label{section:flat_functors}

 Since an ind-abelian category need not have kernels, we cannot ask for a fiber functor to be left exact. The recognition results from \cite[\S\S 7-8]{SCHAEPPI} suggest that the correct generalization of left exact functors to this context is the notion of \emph{flat} functors.

\begin{dfn}\label{dfn:category_of_elements}
 Let $\ca{A}$ be an $R$-linear category, and let $F \colon \ca{A} \rightarrow \Mod_R$ be an $R$-linear functor. The \emph{category of elements} $\el(F)$ of $F$ has objects the pairs $(A,a)$ where $a\in FA$, and morphisms $(A,a) \rightarrow (B,b)$ the morphisms $f \colon A \rightarrow B$ with $Ff(a)=b$.
\end{dfn}

 The equivalence between conditions~(i), (ii) and~(iii) of the following proposition is well-known (see for example \cite[Theorem~3.2]{OBERST_ROHRL}). We will later need the fact that these are also equivalent to condition~(v). The proof that (v) implies (i) is a modification of the usual proof that (iii) implies (i).

\begin{prop}\label{prop:flat_characterization}
 Let $\ca{A}$ be an $R$-linear category with finite direct sums, and let $F \colon \ca{A} \rightarrow \Mod_R$ an $R$-linear functor. Then the following are equivalent:

\begin{enumerate}
 \item[(i)] The category $\el(F)$ is cofiltered;
\item[(ii)] The functor $F$ is a filtered colimit of representable functors;
 \item[(iii)] The left Kan extension $\Lan_Y F$ of $F$ along the Yoneda embedding
\[
 Y \colon \ca{A} \rightarrow \Prs{A}
\]
 is exact;
\item[(iv)] For all reflective subcategories $\ca{C}$ of $\Prs{A}$ which contain the representable presheaves the left Kan extension $\Lan_K F \colon \ca{C} \rightarrow \Mod_R$ is left exact, where
\[
 K \colon \ca{A} \rightarrow \ca{C}
\]
 denotes the corestriction of the Yoneda embedding;
\item[(v)] There exists a reflective subcategory $\ca{C}$ of $\Prs{A}$ such that the Kan extension $\Lan_K F$ of $F$ along the corestricted Yoneda embedding is left exact.
\end{enumerate}
\end{prop}
 
\begin{proof}
 (i) implies (ii). Since $\ca{A}$ has finite direct sums, $F$ is the colimit of the representable functors over $F$. By the Yoneda lemma, this category is equivalent to $\el(F)^{\op}$.

 (ii) implies (iii). This follows from the fact that filtered colimits commute with finite limits (see Lemma~\ref{lemma:flat_lan} for details).

 (iii) implies (iv). By \cite[Theorem~4.47]{KELLY_BASIC}, $\Lan_K F$ is isomorphic to the restriction of $\Lan_Y F$ to $\ca{C}$. Since $\ca{C}$ is reflective, limits in $\ca{C}$ are computed as in $\Prs{A}$.
 
 (iv) implies (v). Obvious.

 (v) implies (i). Suppose that $\ca{C}$ is a reflexive subcategory of $\Prs{A}$, let $K$ be the corestriction of the Yoneda embedding, and assume that $\Lan_K F$ is left exact. We have to show that the category $\el(F)$ of elements of $F$ is cofiltered. 

 The category of elements of $F$ is clearly non-empty, and since $F$ preserves finite direct sums it suffices to check that for any two morphisms $f,g \colon (A,a) \rightarrow (B,b)$ in the category of elements, there exists a morphism $h \colon (C,z) \rightarrow (A,x)$ such that $fh=gh$. If $\ca{A}$ had kernels, we could simply take $C$ to be the kernel of $f-g$.

 Instead we consider the kernel $C^{\prime}$ of $f-g$ in the category $\ca{C}$. The canonical cocone on the diagram $\ca{A} \slash C^{\prime} \rightarrow \Prs{A}$ exhibits $C^{\prime}$ as colimit in the category of all presheaves on $\ca{A}$. This colimit is preserved by the cocontinuous functor $\Lan_Y F$. Since $\Lan_K F$ is isomorphic to the restriction of $\Lan_Y K$ to $\ca{C}$ (see \cite[Theorem~4.47]{KELLY_BASIC}), the colimit in question is also preserved by $\Lan_K F$. It follows in particular that the collection of morphisms
\begin{equation}\label{eqn:jointly_epi_family}
 \Lan_K F (x) \colon \Lan_K F \bigl( \ca{A}(-,A) \bigr) \rightarrow \Lan_K F ( C^{\prime})
\end{equation}
 where $x \in \ca{A} \slash C^{\prime}$ is jointly epimorphic. Since $K$ is fully faithful, the unit 
\[
 \alpha \colon \Lan_K F \cdot K \Rightarrow F 
\]
 of the Kan extension is an isomorphism. Let $a^{\prime}=\alpha_{A}^{-1}(a)$. Since $\Lan_K F$ is left exact, the kernel of $\Lan_K F (f-g)$ is $\Lan_K (C^{\prime})$. From the naturality of $\alpha$ it follows that $a^{\prime} \in \Lan_K (C^{\prime})$. Since the morphisms \eqref{eqn:jointly_epi_family} are jointly epimorphic, there exists a morphism
\[
 x \colon \ca{A}(-,C) \rightarrow C^{\prime}
\]
 and an element $c^{\prime} \in \Lan_K F \bigl(\ca{A}(-,C)\bigr)$ such that $a^{\prime}=\Lan_K F(x)(c^{\prime})$. The composite
\[
 \xymatrix{\ca{A}(-,C) \ar[r]^-{x} & C^{\prime} \ar[r] & \ca{A}(-,A)}
\]
 is of the form $\ca{A}(-,h)$ for a unique $h \colon C \rightarrow A$, and by construction we have $fh=gh$. Let $c=\alpha_C(c^{\prime})$. Using the naturality of $\alpha$ again we find that $F(h)(c)=a$. Thus $h \colon (C,c) \rightarrow (A,a)$ gives the desired morphism in the category $\el(F)$ which equalizes $f$ and $g$. This concludes the proof that $\el(F)$ is cofiltered.
\end{proof}

\begin{dfn}\label{dfn:flat}
 An $R$-linear functor $\ca{A} \rightarrow \Mod_R$ satisfying the equivalent conditions of Proposition~\ref{prop:flat_characterization} is called \emph{flat}. Let $B$ be a commutative $R$-algebra. By abuse of terminology we call an $R$-linear functor $\ca{A} \rightarrow \Mod_B$ \emph{flat} if the composite with the forgetful functor $\Mod_B \rightarrow \Mod_R$ is flat.
\end{dfn}

\subsection{Proof of the recognition theorem for Adams stacks} 
 Theorem~\ref{thm:recognition} is a consequence of the corresponding recognition theorem for Hopf algebroids, which we will prove in the next section.

\begin{thm}\label{thm:recognition_hopf}
 Let $R$ be a commutative ring, and let $\ca{A}$ be a weakly Tannakian $R$-linear ind-abelian category, with $R$-linear fiber functor $w \colon \ca{A} \rightarrow \Mod_B$. Then there exists an Adams Hopf algebroid $(B,\Gamma)$ in $\Mod_R$, together with a symmetric monoidal $R$-linear equivalence $\ca{A} \simeq \Comod_{\fp}(B,\Gamma)$ such that the triangle
\[
 \xymatrix{\ca{A} \ar[rd]_{w} \ar[r]^-{\simeq} & \Comod_{\fp}(B,\Gamma) \ar[d]^{V} \\ & \Mod_B}
\]
 is commutative, where $V$ denotes the forgetful functor.
\end{thm}

\begin{proof}[Proof of Theorem~\ref{thm:recognition}]
 Recall that the category of quasi-coherent sheaves on an algebraic stack is equivalent to the category of comodules of the corresponding flat Hopf algebroid (see \cite[\S 3.4]{NAUMANN} and \cite[Remark~2.39]{GOERSS}).
 If $X$ is an Adams stack, then it is associated to an Adams Hopf algebroid $(B,\Gamma)$. Since the dualizable comodules of an Adams Hopf algebroid form a generator of the category of all comodules (see \cite[Proposition~1.4.1]{HOVEY}), the category $\Comod_{\fp}(B,\Gamma)$ satisfies condition~(ii) of the definition of a weakly Tannakian category. Moreover, since $\Comod(B,\Gamma)$ is a locally finitely presentable abelian category, $\Comod_{\fp}(B,\Gamma)$ is ind-abelian (see Theorem~\ref{thm:ind_abelian}). The forgetful functor is clearly right exact and faithful. It remains to check that it is flat. To do this one can either check by hand that the category of elements is cofiltered. Alternatively we can use the characterization of flat functors given in part~(v) of Proposition~\ref{prop:flat_characterization}.

 Note that the forgetful functor $\Comod(B,\Gamma) \rightarrow \Mod_B$ is the left Kan extension of its restriction to $\Comod_{\fp}(B,\Gamma)$ along the inclusion $\Comod_{\fp}(B,\Gamma) \subseteq \Comod(B,\Gamma)$. Indeed, the forgetful functor is a left adjoint and the subcategory of finitely presentable objects is dense. The claim therefore follows from \cite[Theorem~5.29]{KELLY_BASIC}. Since $(B,\Gamma)$ is a flat Hopf algebroid, the forgetful functor from comodules to $B$-modules is exact. Thus the forgetful functor $\Comod_{\fp}(B,\Gamma) \rightarrow \Mod_B$ is flat by part~(v) of Proposition~\ref{prop:flat_characterization}, with $\ca{C}=\Comod(B,\Gamma)$.

 This concludes the proof that $\Comod_{\fp}(A,\Gamma)$ is an $R$-linear ind-abelian weakly Tannakian category with fiber functor given by the forgetful functor.

 The converse follows from Theorem~\ref{thm:recognition_hopf} as in the proof of \cite[Theorem~1.1.2]{SCHAEPPI_STACKS}.
\end{proof}

\subsection{Proof of the recognition theorem for Hopf algebroids}
 We follow the strategy from \cite[\S 2.3]{SCHAEPPI_STACKS}. We fix a commutative ring $R$, a weakly Tannakian ind-abelian category $\ca{A}$, and a fiber functor $w \colon \ca{A} \rightarrow \Mod_B$. We write $\ca{A}^d$ for the full subcategory consisting of objects with duals.

\begin{prop}\label{prop:flat_for_ind_abelian}
 Let $\ca{A}$ be an essentially small finitely cocomplete category, and let $B$ be a commutative $R$-algebra. Then a functor $F \colon \ca{A} \rightarrow \Mod_B$ is flat if and only if the left Kan extension $L \colon \Ind(\ca{A}) \rightarrow \Mod_B$ of $F$ along the inclusion $\ca{A} \rightarrow \Ind(\ca{A})$ is left exact.
\end{prop}

\begin{proof}
 Since the forgetful functor $\Mod_B \rightarrow \Mod_R$ preserves all colimits, it preserves Kan extensions. It also preserves limits and reflects isomorphisms, so by composing with the forgetful functor we reduce the problem to the case $B=R$. The claim follows from condition~(v) of Proposition~\ref{prop:flat_characterization} for $\ca{C}=\Ind(\ca{A})$.
\end{proof}

 Using this proposition we can prove the analogue of \cite[Proposition~2.3.1]{SCHAEPPI_STACKS}.

\begin{prop}\label{prop:L_comonadic}
 The left Kan extension $L \colon \Ind(\ca{A}) \rightarrow \Mod_B$ of $w$ along the inclusion $\ca{A} \rightarrow \Ind(\ca{A})$ is left exact, comonadic, and strong symmetric monoidal.

 Moreover, the category $\ca{A}^{d}$ is a strong generator of $\Ind(\ca{A})$.
\end{prop}

\begin{proof}
 Since $w$ is right exact, the right adjoint $\widetilde{w}$ of $\Lan_Y w$ factors through $\Ind(\ca{A})$. Since $L$ is isomorphic to the restriction of $\Lan_Y w$ to $\Ind(\ca{A})$ (see \cite[Theorem~4.47]{KELLY_BASIC}), this implies that $L$ is a left adjoint. It is left exact by Proposition~\ref{prop:flat_for_ind_abelian}. Since tensor products of ind-objects are defined as filtered colimits of tensor products in $\ca{A}$, the functor $L$ is symmetric strong monoidal. To show that $L$ is comonadic it remains to check that $L$ is conservative.

 Since $\Ind(\ca{A})$ is abelian, $L$ is conservative if and only if it is faithful, and a left adjoint is faithful if and only if the unit is a monomorphism. The unit at an object $A \in \ca{A}$ is given by
\[
 w_{-,A} \colon \ca{A}(-,A) \rightarrow \Ind\bigl(\ca{A})(w-,w(A)\bigr)=\widetilde{w}\bigl(w(A)\bigr) \smash{\rlap{,}}
\]
 which is a monomorphism by assumption. In any Grothendieck abelian category, filtered colimits of monomorphisms are monomorphisms. Therefore it suffices to prove that $\widetilde{w}$ preserves filtered colimits. Equivalently, we want to show that $w(A)$ is finitely presentable for every $A \in \ca{A}$. 

 By condition~(ii) of the definition of weakly Tannakian categories, there exists an epimorphism $D \rightarrow A$ in $\ca{A}$ such that $D$ has a dual. By definition of ind-abelian categories, this epimorphism is the cokernel of some morphism $A^{\prime} \rightarrow D$. Again using condition~(ii) we get an epimorphism $D^{\prime} \rightarrow A^{\prime}$ such that $D^{\prime}$ has a dual. Thus we have a right exact sequence
\begin{equation}\label{eqn:cokernel_of_duals}
 \xymatrix{D^{\prime} \ar[r] & D \ar[r] & A \ar[r] & 0} 
\end{equation}
 in $\ca{A}$. This exact sequence is preserved by the right exact functor $w$. Moreover, since $w$ is strong monoidal, it sends $D$ and $D^{\prime}$ to dualizable $B$-modules, which are precisely the finitely generated projective $B$-modules. Thus $w(A)$ is indeed finitely presentable. This concludes the proof that $L$ is faithful, hence that $L$ is comonadic.

 It remains to check that the category $\ca{A}^d$ of duals forms a strong generator. By \cite[Proposition~3.40]{KELLY_BASIC} it suffices to check that $\Ind(\ca{A})$ is the closure of the full subcategory $\ca{A}^{d}$ under colimits. Every object in $\Ind(\ca{A})$ is a filtered colimit of objects in $\ca{A}$. Since the inclusion of $\ca{A} \rightarrow \Ind(\ca{A})$ is right exact, the claim follows from the existence of the right exact sequence~\eqref{eqn:cokernel_of_duals}.
\end{proof}

 Using the second part of the above proposition we can show that \cite[Proposition~2.5.1]{SCHAEPPI_STACKS} holds for ind-abelian weakly Tannakian categories.

\begin{prop}\label{prop:duals_dense}
 Let $L$ be the left Kan extension of $w \colon \ca{A} \rightarrow \Mod_B$ along the inclusion $\ca{A} \rightarrow \Ind(\ca{A})$. Write $K \colon \ca{A}^{d} \rightarrow \Ind(\ca{A})$ for the evident inclusion. The functor $K$ is dense, and $L$ is the left Kan extension of the restriction of $w$ to $\ca{A}^d$ along $K$.
\end{prop}

\begin{proof}
 From Proposition~\ref{prop:L_comonadic} we know that $\ca{A}^d$ is a strong generator of $\Ind(\ca{A})$. Therefore it is dense by \cite[Theorem~2 and Example~(3)]{DAY_STREET_GENERATORS}.

 Since $L$ is comonadic (see Proposition~\ref{prop:L_comonadic}), it preserves all colimits. Thus it is the left Kan extension of $LK$ along $K$ by \cite[Theorem~5.29]{KELLY_BASIC}. The claim follows from the fact that the restriction of $L$ to $\ca{A}$ is isomorphic to $w$.
\end{proof}

\begin{cor}\label{cor:comonad_from_duals}
 Let $w_d$ denote the restriction of $w$ to $\ca{A}_d$. The comonad induced by the functor $L$ is naturally isomorphic (as a symmetric monoidal $R$-linear comonad) to the comonad induced by $\Lan_Y w_d \dashv \widetilde{w_d}$.
\end{cor}

\begin{proof}
 The proof follows verbatim the proof of \cite[Corollary~2.5.2]{SCHAEPPI_STACKS}.   
\end{proof}

\begin{cor}\label{cor:comonad_Hopf_monoidal}
 The symmetric monoidal comonad induced by the functor $L$ from Proposition~\ref{prop:L_comonadic} is cocontinuous and Hopf monoidal.
\end{cor}

\begin{proof}
 The proof follows verbatim the proof of \cite[Corollary~2.5.3]{SCHAEPPI_STACKS}.
\end{proof}

\begin{proof}[Proof of Theorem~\ref{thm:recognition_hopf}]
 From Proposition~\ref{prop:L_comonadic} we know that $\Ind(\ca{A})$ is equivalent to the category of comodules of the comonad induced by $L$. Using Corollary~\ref{cor:comonad_Hopf_monoidal}, we find that this comonad is induced by a flat Hopf algebroid $(B,\Gamma)$. It only remains to show that $(B,\Gamma)$ is an Adams Hopf algebroid. Since the category of dual comodules forms a generator of $\Ind(\ca{A})\simeq \Comod(B,\Gamma)$ (see Proposition~\ref{prop:L_comonadic}), this follows from \cite[Theorem~1.3.1]{SCHAEPPI_STACKS}.
\end{proof}

%From \cite{SCHAEPPI_STACKS} we know that the pseudofunctor
%\[
% \QCoh(-) \colon \ca{AS}^{\op} \rightarrow \ca{RM}
%\]
% from Adams stacks to right exact symmetric monoidal categories is an equivalence on hom-categories. 
\section{Binary products and fiber products of Adams stacks}\label{section:adams}

 The proof of Theorem~\ref{thm:adams_fiber_product} proceeds in two steps. First, we prove that Adams stacks are closed under finite products.

\begin{thm}\label{thm:product_adams}
 If $X$ and $Y$ are Adams stacks, then $X\times Y$ is an Adams stack.
\end{thm}

 As an immediate corollary, we find that schemes with the resolution property are closed under products.

\begin{cor}
 The class of quasi-compact semi-separated schemes with the strong resolution property is closed under finite products.
\end{cor}

\begin{proof}
 This follows from Theorem~\ref{thm:product_adams} and the fact that the embedding of schemes in stacks preserves products.
\end{proof}

 In order to prove Theorem~\ref{thm:product_adams}, it is convenient to use the equivalence between quasi-coherent sheaves on an algebraic stack $X$ and comodules of the corresponding Hopf algebroid. 

 Recall that a Hopf algebroid $(A,\Gamma)$ is called an \emph{Adams Hopf algebroid} if $\Gamma$, considered as an $(A,\Gamma)$-comodule, is a filtered colimit of dualizable comodules $\Gamma_i$. Theorem~\ref{thm:product_adams} is a consequence of the following theorem from \cite{SCHAEPPI_STACKS}.

\begin{thm}\label{thm:adams_characterization}
 An algebraic stack has the strong resolution property if and only if it is associated to an Adams Hopf algebroid.
\end{thm}

\begin{proof}
 See \cite[Theorem~1.3.1]{SCHAEPPI_STACKS}.
\end{proof}

\begin{proof}[Proof of Theorem~\ref{thm:product_adams}]
 Since the associated stack functor preserves products, it suffices to show that groupoids corresponding to Adams Hopf algebroids are closed under finite products. Equivalently, we need to show that Adams Hopf algebroids are closed under finite coproducts. 

 If $(A,\Gamma)$ and $(B,\Sigma)$ are Adams Hopf algebroids, with $\Gamma$ a filtered colimit of dualizable comodules $\Gamma_i$ and $\Sigma$ a filtered colimit of dualizable comodules $\Sigma_j$, then $\Gamma \ten{R} \Sigma$ is the filtered colimit of the $(A\ten{R} B,\Gamma \ten{R} \Sigma)$-comodules $\Gamma_i \ten{R} \Sigma_j$. The conclusion follows from Theorem~\ref{thm:adams_characterization}.
\end{proof}

\subsection{Fiber products}
 In this section we will conclude the proof that the class of algebraic stacks with the resolution property is closed under (bicategorical) fiber products. Our strategy is to reduce the problem to Theorem~\ref{thm:product_adams}. We illustrate this for the case schemes. If $X$, $Y$ are quasi-compact semi-separated schemes with the resolution property, and $Z$ is a separated scheme, then the fiber product $X\pb{Z} Y$ is a closed subscheme of $X \times Y$. Since all closed subschemes of a scheme with the resolution property have the resolution property it follows that $X \pb{Z} Y$ has the resolution property. In this section we show that this argument generalizes to the case where $X$, $Y$, and $Z$ are algebraic stacks.

\begin{lemma}\label{lemma:thomasons_trick}
 Let $(A,\Gamma)$ be an Adams Hopf algebroid, and let $(B,\Sigma)$ be any flat Hopf algebroid. Let
\[
 F \colon \Comod(A,\Gamma) \rightarrow \Comod(B,\Sigma)
\]
 be a symmetric strong monoidal $R$-linear functor with right adjoint $G$. If the counit $\varepsilon \colon FG \Rightarrow \id$ is an epimorphism, then $(B,\Sigma)$ is an Adams Hopf algebroid.
\end{lemma}

\begin{proof}
 The argument we give is an adaption of an argument due to Thomason (see the proof of \cite[Lemma~2.6]{THOMASON}). 

 Fix a $(B,\Sigma)$-comodule $M$. Since $(A,\Gamma)$ has the resolution property, we can find dualizable $(A,\Gamma)$-comodules $N_i$, together with an epimorphism
\[
 \xymatrix{ \bigoplus_{i\in I} N_i \ar[r] & G(M) }
\]
 of comodules. Since $F$ preserves duals and epimorphism, the composite
\[
 \xymatrix{ \bigoplus_{i\in I} F(N_i) \ar[r] & FG(M) \ar[r]^-{\varepsilon} & M }
\]
 gives the desired epimorphism of $(B,\Sigma)$-comodules.
\end{proof}

\begin{lemma}\label{lemma:discrete_fibration}
 Let $(f_0,f_1) \colon (A,\Gamma) \rightarrow (B,\Sigma)$ be a morphism of Hopf algebroids. If the diagram
\[
 \xymatrix{ A \ar[r]^{f_0} \ar[d]_{\sigma} & B \ar[d]^{\sigma} \\ \Gamma \ar[r]_{f_1} & \Sigma}
\]
 (where $\sigma$ denotes the morphism representing the source) is a pushout diagram, then the counit of the induced left adjoint
\[
 \Comod(A,\Gamma) \rightarrow \Comod(B,\Sigma)
\]
 is an epimorphism.
\end{lemma}

\begin{proof}
 We use the adjoint lifting theorem for comonads dual to \cite[Theorem~4]{JOHNSTONE}.

 The assumption that $(f_0,f_1)$ fits in the above pushout diagram implies that the 2-cell part $\lambda$ of the induced morphism
\[
 \xymatrix{\Mod_A \ar[r]^{(f_0)_{\ast}} \dtwocell\omit{^<-6>\lambda} \ar[d]_{\Gamma \ten{A} - }  &  \Mod_B \ar[d]^{\Sigma \ten{B} - } \\
 \Mod_A \ar[r]_{(f_0)_{\star}} & \Mod_B}
\]
 of comonads is invertible. Thus the mate of $\lambda^{-1}$ under the adjunction $(f_0)_{\ast} \dashv U$ in the sense of \cite{KELLY_STREET} endows $U$ with the structure of a morphism of comonads. Moreover, the counit $\varepsilon \colon (f_0)_{\ast} U \Rightarrow \id$ is a 2-cell of comonad morphisms by construction. This shows that the counit of the induced adjunction between the categories of comodules is obtained by lifting the counit of $(f_0)_{\ast} \dashv U$.

 Since $U$ is faithful, the counit $\varepsilon \colon (f_0)_{\ast} U \Rightarrow \id$ is an epimorphism. The conclusion follows from the fact that a morphism of comodules is an epimorphism if and only if its underlying morphism of modules is an epimorphism.
\end{proof}

\begin{prop}\label{prop:affine_morphism}
 Let $f \colon X \rightarrow Y$ be an affine morphism of algebraic stacks. Then the counit of the induced adjunction
\[
 f^{\ast} \colon \QCoh(Y) \rightleftarrows \QCoh(X) \colon f_{\ast} 
\]
 is an epimorphism. Moreover, if $Y$ is an Adams stack, then so is $X$. 
\end{prop}

\begin{proof}
 Let $p \colon \Spec(A) \rightarrow Y$ be a faithfully flat morphism, and let $\Spec(\Gamma)= \Spec(A) \pb{Y} \Spec(A)$. Since $f$ is affine, the pullback of $p$ along $f$ is given by a faithfully flat morphism $p^{\prime} \colon \Spec(B) \rightarrow X$. Let $\Spec(\Sigma)=\Spec(B) \pb{X} \Spec(B)$.

 By construction, $(A,\Gamma)$ and $(B,\Sigma)$ are Hopf algebroids, with associated stacks $Y$ respectively $X$ (see \cite[\S 3.3]{NAUMANN}). Moreover, the pullback of $f$ gives a morphism of Hopf algebroids $(f_0,f_1) \colon (B,\Sigma) \rightarrow (A,\Gamma)$ whose induced morphism between associated stacks is $f$. The cancellation law for pullbacks shows that
\[
 \xymatrix{\Spec(\Sigma) \ar[d]_{\Spec(f_1)} \ar[r]^{\Spec(\sigma)} & \Spec(B) \ar[d]^{\Spec(f_0)} \\ \Spec(\Gamma) \ar[r]_{\Spec(\sigma)} & \Spec(A) }
\]
 is a pullback diagram. Passing along the equivalence between affine schemes over $R$ and commutative $R$-algebras, this shows that $(f_0,f_1)$ satisfies the condition of Lemma~\ref{lemma:discrete_fibration}. Therefore the counit of the induced adjunction between the categories of comodules is an epimorphism. By Lemma~\ref{lemma:thomasons_trick} it follows that $\Comod(B,\Sigma)$ is generated by duals if $\Comod(A,\Gamma)$ is. 

Using the natural equivalences
\[
 \Comod(A,\Gamma) \simeq \QCoh(Y) \quad\text{and}\quad \Comod(B,\Sigma) \simeq \QCoh(X)
\]
 (see \cite[Remark~2.39]{GOERSS} and \cite[\S 3.4]{NAUMANN}) we find that the counit of the adjunction $f^{\ast} \dashv f_{\ast}$ is an epimorphism, and that $X$ has the strong resolution property if $Y$ does.
\end{proof}

\begin{cor}\label{cor:adams_fiber_product}
 Let $X$ and $Y$ be Adams stacks, and let $Z$ be an algebraic stack. Then the pullback $X \pb{Z} Y$ of $f \colon X \rightarrow Z$ and $g \colon Y \rightarrow Z$ is an Adams stack for all morphisms of stacks $f$ and $g$.
\end{cor}

\begin{proof}
 By Theorem~\ref{thm:product_adams}, the product $X\times Y$ is an Adams stack, so by Proposition~\ref{prop:affine_morphism} it suffices to show that the morphism
\[
 X\pb{Z} Y \rightarrow X\times Y
\]
 is affine. Since this morphism fits in a bicategorical pullback square
\[
 \xymatrix{ X\pb{Z} Y \ar[d]  \ar[r] & Z \ar[d]^{\Delta} \\ X\times Y \ar[r]_-{f\times g} & Z\times Z}
\]
 it is obtained by pulling back the diagonal of $Z$. The claim follows from the fact that the diagonal of an algebraic stack is affine.
\end{proof}

\begin{proof}[Proof of Theorem~\ref{thm:adams_fiber_product}]
 The category of Adams stacks is closed under fiber products by Corollary~\ref{cor:adams_fiber_product}.
\end{proof}

\section{Bicategorical coproducts of right exact symmetric monoidal categories}\label{section:rex_coproduct}
 Recall that a symmetric monoidal $R$-linear category is \emph{right exact symmetric monoidal} if it is finitely cocomplete and the tensor product is right exact in each variable. In this section we will prove the following theorem.

\begin{thm}\label{thm:rex_tensor_coproduct}
 Let $\ca{A}$ and $\ca{B}$ be right exact symmetric monoidal $R$-linear categories. Then the right exact tensor product $\ca{A}\boxtimes \ca{B}$ is a bicategorical coproduct in the 2-category of right exact symmetric monoidal $R$-linear categories.
\end{thm}

 First note that the 2-category $\Cat_R$ of $R$-linear categories and the 2-category $\Rex$ of finitely cocomplete $R$-linear categories and right exact $R$-linear functors both are strict symmetric $\Cat$-multicategories. The multi-hom-categories are given by the category
\[
 [\ca{A}_1,\ldots, \ca{A}_n ; \ca{B}]
\]
 of multi-functors that are $R$-linear in each variable, and the category
\[
 \Rex [\ca{A}_1,\ldots, \ca{A}_n ; \ca{B}]
\]
 of multi-functors that are $R$-linear and right exact in each variable respectively. The forgetful functor
\[
 \Rex \rightarrow \Cat_R 
\]
 is clearly a functor of strict symmetric $\Cat$-multicategories.

 Moreover, we have natural \emph{isomorphisms}
\begin{equation}\label{eqn:tensor}
 [\ca{A},\ca{B} ; \ca{C}] \cong [\ca{A}\otimes \ca{B},\ca{C}]
\end{equation}
 and natural \emph{equivalences}
\begin{equation}\label{eqn:rex_tensor}
 \Rex [\ca{A},\ca{B} ; \ca{C}] \simeq \Rex [\ca{A}\boxtimes \ca{B},\ca{C}]
\end{equation}
 by the universal properties of $\ca{A}\otimes \ca{B}$ and $\ca{A} \boxtimes \ca{B}$. From~\eqref{eqn:tensor} and \eqref{eqn:rex_tensor} we get a symmetric monoidal 2-category structure on $\Cat_R$ and $\Rex$. Note that structure on $\Cat_R$ is stricter. For example, the associator in $\Cat_R$ is an isomorphism and the pentagon commutes strictly, whereas in $\Rex$, the associator is only an equivalence, and the pentagon only commutes up to coherent natural isomorphism. This stems from the fact that~\eqref{eqn:tensor} is an isomorphism, while~\eqref{eqn:rex_tensor} is only an equivalence.

 A right exact symmetric monoidal $R$-linear category is precisely a symmetric pseudomonoid in the symmetric monoidal bicategory $(\Rex,\boxtimes)$, in the sense of \cite[\S 4]{MCCRUDDEN_BALANCED}. Since the pseudofunctor
\[
 -\boxtimes - \colon \Rex \times \Rex \rightarrow \Rex
\]
 is symmetric monoidal (cf.\ \cite[Proposition~16 and Definition~18]{DAY_STREET}) it follows that the tensor product $\ca{A}\boxtimes \ca{B}$ of two right exact symmetric monoidal $R$-linear categories is a right exact symmetric monoidal $R$-linear category in a natural way. Therefore Theorem~\ref{thm:rex_tensor_coproduct} is a consequence of the following theorem.

\begin{thm}\label{thm:pseudomonoid_coproduct}
 Let $(\ca{M},\otimes,I)$ be a symmetric monoidal bicategory, and let $A_0$ and $A_1$ be symmetric pseudomonoids in $\ca{M}$, with product $p_k \colon A_k\otimes A_k \rightarrow A_k$ and unit $i_k \colon I \rightarrow A_k$, $k=0,1$. Then the morphisms
\[
 \xymatrix{ A_0 \simeq A_0\otimes I \ar[r]^-{A_0 \otimes i_1} & A_0 \otimes A_1 } \quad \text{and} \quad
 \xymatrix{ A_1 \simeq I\otimes A_1 \ar[r]^-{i_0 \otimes A_1} & A_0 \otimes A_1 }
\]
 are symmetric strong monoidal, and they exhibit $A_0 \otimes A_1$ as bicategorical coproduct of $A_0$ and $A_1$ in the bicategory of symmetric pseudomonoids in $\ca{M}$. 
\end{thm}

 We defer the proof of this theorem to Appendix~\ref{section:pseudomonoid_coproduct}.

\begin{proof}[Proof of Theorem~\ref{thm:rex_tensor_coproduct}]
 As we have just observed, a right exact symmetric monoidal category is precisely a symmetric pseudomonoid in the symmetric monoidal 2-category $(\Rex, \boxtimes)$. The claim follows immediately from Theorem~\ref{thm:pseudomonoid_coproduct}.
\end{proof}

 Using the above construction of the symmetric monoidal 2-categories $(\Cat_R,\otimes)$ and $(\Rex,\boxtimes)$ we can also show that
\[
 Z \colon \ca{A} \otimes \ca{B} \rightarrow \ca{A} \boxtimes \ca{B}
\]
 is monoidal.

\begin{prop}\label{prop:forgetful_monoidal}
 The forgetful functor $(\Rex,\boxtimes) \rightarrow (\Cat_R,\otimes)$ is symmetric strong monoidal, with structure map isomorphic to the functor
\[
 Z \colon \ca{A} \otimes \ca{B} \rightarrow \ca{A} \boxtimes \ca{B} 
\]
 from Proposition~\ref{prop:rex_tensor}.
\end{prop}

\begin{proof}
 The symmetric monoidal structures on $\Cat_R$ and $\Rex$ are defined by transfer of the respective strict symmetric $\Cat$-multicategory structures along natural isomorphisms, respectively natural equivalences. The forgetful functor
\[
 \Rex \rightarrow \Cat_R
\]
 is strictly compatible with the symmetric $\Cat$-multicategory structure. Thus $U$ becomes symmetric monoidal by transfer of structure. The structure morphism induces the composite
\[
 \Rex[\ca{A}\boxtimes \ca{B},\ca{C}] \simeq \Rex[\ca{A},\ca{B}; \ca{C}] \rightarrow [\ca{A},\ca{B}; \ca{C}] \cong [\ca{A}\otimes \ca{B},\ca{C}] \smash{\rlap{,}}
\]
 which implies that it has the same universal property as $Z$.
\end{proof}

\begin{rmk}\label{rmk:pseudocommutative_monad}
 The above proposition can also be seen as a special case of results of Hyland and Power \cite{HYLAND_POWER} and Franco L\'opez \cite{IGNACIO_PSEUDO_COMMUTATIVE}. Namely, $\Rex$ is the category algebras for a 2-monad of the KZ type, which is symmetric pseudo-commutative by \cite[Theorem~7.3 and Corollary~7.9]{IGNACIO_PSEUDO_COMMUTATIVE}. The induced adjunction between algebras and their underlying objects is symmetric monoidal by \cite[Theorem~13]{HYLAND_POWER}, generalized to the enriched context considered in \cite{IGNACIO_PSEUDO_COMMUTATIVE}.
\end{rmk}

 As a corollary we find that $Z$ is strong monoidal. This implies that $Z$ preserves duals, which we will use to show that the right exact tensor product of two weakly Tannakian categories has enough duals.

\begin{cor}\label{cor:Z_strong_monoidal}
 The forgetful functor $\Rex \rightarrow \Cat_R$ lifts to the categories of pseudomonoids, and the functor
\[
  Z \colon \ca{A} \otimes \ca{B} \rightarrow \ca{A} \boxtimes \ca{B} 
\]
 from Proposition~\ref{prop:rex_tensor} is strong monoidal.
\end{cor}

\begin{proof}
 Symmetric monoidal pseudofunctors lift to the categories of pseudomonoids by \cite[Proposition~15]{DAY_STREET}. Moreover, the pseudonatural transformation
\[
 \xymatrix{\Rex \dtwocell\omit{^<-9> Z} \times \Rex \ar[r]^-{\boxtimes} \ar[d] & \Rex \ar[d] \\ \Cat_R \times \Cat_R \ar[r]_-{\otimes} & \Cat_R}
\]
 is a monoidal pseudonatural transformation (see \cite[p.~125]{DAY_STREET}). Therefore it lifts to the category of pseudomonoids, that is, it is naturally endowed with the structure of a strong monoidal functor.
\end{proof}

 We will use the following proposition to show that the right exact tensor product of two fiber functors is a fiber functor.

\begin{prop}\label{prop:tensor_of_modules}
 Let $A$ and $B$ be two commutative $R$-algebras. Then there is an equivalence
\[
 \Mod_A^{\fp} \boxtimes \Mod_B^{\fp} \simeq \Mod_{A\otimes B}^{\fp}
\]
 of symmetric monoidal $R$-linear categories. Moreover, the composite
\[
\xymatrix{ \Mod_A^{\fp} \otimes \Mod_B^{\fp} \ar[r]^-{Z} & \Mod_A^{\fp} \boxtimes  \Mod_B^{\fp} \ar[r]^-{ \simeq } & \Mod_{A\otimes B}^{\fp} }
\]
 is naturally isomorphic to the functor which sends $(M,N)$ to $M\otimes N$.
\end{prop}

\begin{proof}
 First note that for a right exact symmetric monoidal category $\ca{C}$, there is a natural equivalence
\[
 \SymPsMon(\Cat_R)[A,\ca{C}] \simeq \SymPsMon(\Rex)[\Mod_A^{\fp},\ca{C}]
\]
 between the category of symmetric strong symmetric monoidal functors $A \rightarrow \ca{C}$ and the category of right exact strong symmetric monoidal functors $\Mod_A^{\fp} \rightarrow \ca{C}$. This follows from the abstract framework developed by Hyland, Power and L\'opez Franco mentioned in Remark~\ref{rmk:pseudocommutative_monad}. It can also be seen directly, following the proof of \cite[Theorem~5.1]{IM_KELLY}.

 Applying this and Theorem~\ref{thm:pseudomonoid_coproduct} twice, we get an equivalence
\begin{align*}
 \SymPsMon&(\Rex)[\Mod_A^{\fp}\boxtimes \Mod_B^{\fp},\ca{C}] \\
& \simeq \SymPsMon(\Rex)[\Mod_A^{\fp},\ca{C}] \times \SymPsMon(\Rex)[\Mod_B^{\fp},\ca{C}] \\
&\simeq \SymPsMon(\Cat_R)[A,\ca{C}] \times \SymPsMon(\Cat_R)[B,\ca{C}] \\
& \simeq \SymPsMon(\Cat_R)[A\otimes B, \ca{C}] \\
& \simeq \SymPsMon(\Rex)[\Mod_{A\otimes B}^{\fp},\ca{C}]
\end{align*}
 which is natural in $\ca{C}$. By Yoneda we get the desired equivalence
\[
\Mod_A^{\fp} \boxtimes \Mod_B^{\fp} \simeq \Mod_{A\otimes B}^{\fp}
\]
 of symmetric monoidal categories. Note that what we have proved is a special case of the fact that a left biadjoint functor preserves bicategorical coproducts. In particular, the above equivalence is given by the essentially unique strong symmetric monoidal functor which is compatible with the two respective coproduct inclusions.

 To finish the proof it suffices to check that the composite
\[
 \xymatrix{\Mod_A^{\fp} \otimes \Mod_A^{\fp} \ar[r]^{Z} & \Mod_A^{\fp} \boxtimes \Mod_A^{\fp} \ar[r]^-{\simeq} & \Mod_{A\otimes B}^{\fp}}
\]
 is isomorphic to the strong symmetric monoidal functor
\[
 \Mod_A^{\fp} \otimes \Mod_A^{\fp} \rightarrow \Mod_{A\otimes B}^{\fp} 
\]
 which sends $(M,N)$ to $M\otimes N$. Since the domain is a bicategorical coproduct in the category of symmetric monoidal $R$-linear categories and strong symmetric monoidal $R$-linear functors (see Theorem~\ref{thm:pseudomonoid_coproduct}), it suffices to check that the two composites with the coproduct inclusions are naturally isomorphic. We have already noted that the equivalence
\[
 \Mod_A^{\fp} \boxtimes \Mod_B^{\fp} \simeq \Mod_{A\otimes B}^{\fp}
\]
 is compatible with the coproduct inclusions by construction. The functor $Z$ commutes with the coproduct inclusions (up to isomorphism) because the forgetful functor
\[
 \Rex \rightarrow \Cat_R
\]
 is symmetric monoidal (see Proposition~\ref{prop:forgetful_monoidal}).
\end{proof}

\section{Tensor products of weakly Tannakian categories}\label{section:kelly_tensor}

 The goal of this section is to show that weakly Tannakian categories are closed under the Kelly tensor product. Throughout this section, we fix two weakly Tannakian $R$-linear categories $\ca{A}$ and $\ca{B}$, with fiber functors
\[
 w\colon \ca{A} \rightarrow \Mod_A \quad \text{and} \quad v \colon \ca{B} \rightarrow \Mod_B
\]
 respectively.

 Recall from \S \ref{section:rex_coproduct} that a right exact monoidal category is precisely a pseudomonoid in the monoidal 2-category $(\Rex,\boxtimes)$. This implies that $\ca{A} \boxtimes \ca{B}$ is a right exact symmetric monoidal category, and that
\[
 w\boxtimes v \colon \ca{A} \boxtimes \ca{B} \rightarrow \Mod_A^{\fp} \boxtimes \Mod_B^{\fp}
\]
 is a strong symmetric monoidal right exact functor.

 From Proposition~\ref{prop:tensor_of_modules} we know that $\Mod_A^{\fp} \boxtimes  \Mod_A^{\fp} \simeq  \Mod_{A\otimes B}^{\fp}$. Therefore $w \boxtimes v$ is a natural candidate for a fiber functor of $\ca{A} \boxtimes \ca{B}$. It is right exact by construction, so we only need to check that it is faithful and flat. In order to do this we use the construction of the right exact tensor product described in Proposition~\ref{prop:rex_tensor}. In addition to the notation introduced there we also use the abbreviation
\[
 \Lex_{\Sigma} = \Lex_{\Sigma}[(\ca{A}\otimes \ca{B})^{\op},\Mod_R]
\]
 and we write
\[
 \xymatrix{ \ca{A} \boxtimes \ca{B} \ar[r]^-{J} & \Lex_{\Sigma} \ar[r]^-{I} & [(\ca{A}\otimes \ca{B})^{\op},\Mod_R] }
\]
 for the natural inclusions. Given a commutative $R$-algebra $A$, we let
\[
U \colon \Mod_A \rightarrow \Mod_R  \smash{\rlap{,}}
\]
 be the forgetful functor. We let
\[
 L \colon  \Lex_{\Sigma} \rightarrow \Mod_R 
\]
 be the restriction of $\Lan_Y Uw\otimes Uv$ to $\Lex_{\Sigma}$.

 We briefly outline the proof strategy. Instead of proving that $w\boxtimes v$ is faithful directly, we give an expression of $w\boxtimes v$ in terms of the more familiar functor $L$ (see Lemmas~\ref{lemma:L_left_adjoint} and~\ref{lemma:fiber_functor_lan}). We then use the explicit formula for computing Kan extensions from \cite{KELLY_BASIC} to show that $L$ is conservative and exact (see Lemmas~\ref{lemma:flat_lan} and~\ref{lemma:L_lex_conservative}). This is the key fact needed for showing that $\ca{A}\boxtimes \ca{B}$ is ind-abelian and that $w\boxtimes v$ is a fiber functor.

\begin{lemma}\label{lemma:L_left_adjoint}
 There is a natural isomorphism $\Lan_Y Uw\otimes Uv \cong L R$, and the functor $L$ is a left adjoint.
\end{lemma}

\begin{proof}
 It is a general fact that if a right adjoint factors through a reflective subcategory, then the left adjoint sends the unit of the reflection to an isomorphism. This follows for example from the fact that the unit whiskered with the left adjoint is the mate of an identity natural transformation.

 Thus, to show that
\[
(\Lan_Y Uw \otimes Uv) \eta \colon \Lan_Y Uw \otimes Uv \Rightarrow (\Lan_Y Uw \otimes Uv) IR=L R
\]
 is an isomorphism, it suffices to check that the right adjoint of $\Lan_Y Uw \otimes Uv$ factors through $\Lex_{\Sigma}$. This follows from the fact that $Uw \otimes Uv$ sends sequences in $\Sigma$ to right exact sequences. %, and it also implies that $L$ is a left adjoint.
\end{proof}

\begin{lemma}\label{lemma:fiber_functor_lan}
 The diagram
\[
\xymatrix{\ca{A} \boxtimes \ca{B} \ar[rr]^-{J} \ar[d]_{w\boxtimes v} && \Lex_{\Sigma}  \ar[d]^{L} \\ \Mod_A^{\fp} \boxtimes \Mod_B^{\fp} \ar[r]^-{\simeq} & \Mod_{A\otimes B}^{\fp} \ar[r]^{U} & \Mod_R} 
\]
 commutes up to natural isomorphism.
\end{lemma}

\begin{proof}
 Since all the functors involved are right exact, we can use the universal property of $Z \colon \ca{A}\otimes \ca{B} \rightarrow \ca{A} \boxtimes \ca{B}$, that is, it suffices to check that the outer rectangle of the diagram
\[
\xymatrix@!C=70pt{\ca{A} \otimes \ca{B} \ar[r]^{Z} \ar[d]_{w\otimes v} \ar@{}[rd]|{\cong} & \ca{A} \boxtimes \ca{B} \ar[rr]^-{J} \ar[d]_{w\boxtimes v} && \Lex_{\Sigma} \ar[d]^{L} \\ 
 \Mod_A^{\fp} \otimes \Mod_B^{\fp} \ar[r]^{Z} & \Mod_A^{\fp} \boxtimes \Mod_B^{\fp} \ar[r]^-{\simeq} & \Mod_{A\otimes B}^{\fp} \ar[r]^{U} & \Mod_R}  
\]
 commutes up to isomorphism (cf.\ \cite[Theorem~6.23]{KELLY_BASIC}). By Lemma~\ref{lemma:L_left_adjoint} we have
\[
 LJZ=LRY\cong (\Lan_Y Uw\otimes Uv) Y \cong Uw\otimes Uv \smash{\rlap{.}}
\]
 Using Proposition~\ref{prop:tensor_of_modules} we find that the lower composite is isomorphic to $Uw\otimes Uv$ as well.
\end{proof}

 \begin{lemma}\label{lemma:flat_lan}
 Let $\ca{C}$ and $\ca{D}$ be a small $R$-linear categories. If $F \colon \ca{C} \rightarrow \Mod_R$ is a filtered colimit
\[
 F \cong \colim^{i \in \ca{I}} \ca{C}(c_i,-)
\]
 of representable functors, then there is an isomorphism
\begin{equation}\label{eqn:flat_lan}
 \Lan_Y F (H) \cong \colim^{i \in \ca{I}} H(c_i) 
\end{equation}
 natural in $H \in [\ca{C}^{\op},\Mod_R]$. If $G \colon \ca{D} \rightarrow \Mod_R$ is also given by a filtered colimit
\[
 G \cong \colim^{j \in \ca{J}} \ca{D}(d_j,-)
\]
 of representable functors, then there is an isomorphism
\begin{equation}\label{eqn:flat_tensor_lan}
 \Lan_Y F\otimes G (H) \cong \colim^{(i,j) \in \ca{I} \times \ca{J}} H(c_i,d_j) 
\end{equation}
 natural in $H \in [(\ca{C} \otimes \ca{D})^{\op},\Mod_R]$.
 \end{lemma}

\begin{proof}
 Since $\ca{C}$ is small, we have
\[
 \Lan_Y F(H)\cong \int^{\ca{C}} H(c) \otimes F(c)
\]
 (see \cite[Formulas~(4.31) and (3.70)]{KELLY_BASIC}). Using the assumption on $F$, the fact that colimits commute with each other, and the Yoneda lemma, we find that
\begin{align*}
 \int^{\ca{C}} H(c) \otimes F(c) & \cong \int^{\ca{C}} G(c) \otimes \colim^{i \in \ca{I}} \ca{C}(c_i,c) \\  
 & \cong  \colim^{i \in \ca{I}} \int^{\ca{C}} H(c) \otimes \ca{C}(c_i,c) \\
 & \cong \colim^{i \in \ca{I}} H(c_i) \smash{\rlap{,}}
\end{align*}
 which shows that \eqref{eqn:flat_lan} holds. Similarly we find that
\begin{align*}
 \Lan_Y F\otimes G (H) & \cong \int^{\ca{C} \otimes \ca{D}} H(c,d) \otimes F(c) \otimes G(d) \\
 &\cong \int^{\ca{C} \otimes \ca{D}} H(c,d) \otimes \colim^{i \in \ca{I}} \ca{C}(c_i,c) \otimes \colim^{j \in \ca{J}} \ca{D}(d_j,d) \\
 &\cong \colim^{(i,j)\in \ca{I} \times \ca{J}} \int^{\ca{C}\otimes \ca{D}} H(c,d) \otimes \ca{C}(c_i,c) \otimes \ca{D}(d_j,d) \smash{\rlap{.}}
\end{align*}
 Using the Fubini Theorem for coends and the Yoneda lemma we get
\begin{align*}
 \int^{\ca{C}\otimes \ca{D}} H(c,d) \otimes \ca{C}(c_i,c) \otimes \ca{D}(d_i,d)  
& \cong \int^{\ca{C}} \int^{\ca{D}} H(c,d) \otimes \ca{C}(c_i,c) \otimes \ca{D}(d_j,d) \\
&\cong \int^{\ca{C}} H(c,d_j) \otimes \ca{C}(c_i,c) \\
&\cong H(c_i,d_j) \smash{\rlap{,}}
\end{align*}
which shows that \eqref{eqn:flat_tensor_lan} holds.
\end{proof}

\begin{lemma}\label{lemma:L_lex_conservative}
 The left adjoint $L \colon \Lex_{\Sigma} \rightarrow \Mod_R$ is exact and conservative.
\end{lemma}

\begin{proof}
 Since $Uw$ and $Uv$ are flat, they can be written as filtered colimits
\[
 Uw \cong \colim^{i\in \ca{I}} \ca{A}(a_i,-) \quad \text{and} \quad
 Uv \cong \colim^{j\in \ca{J}} \ca{B}(b_j,-)
\]
 of representable functors.
 
 Since limits in the reflexive subcategory $\Lex_{\Sigma}$ are computed as in the category of presheaves, exactness of $L$ follows from \eqref{eqn:flat_tensor_lan} and the fact that filtered colimits commute with finite limits in the category of presheaves.

 Since $w$ and $v$ are faithful, the restrictions of the left Kan extensions $\Lan_Y w$ and $\Lan_Y v$ to $\Ind(\ca{A})$ and $\Ind(\ca{B})$ are conservative. It follows that both $\Lan_Y Uw$ and $\Lan_Y Uv$ are conservative when restricted to ind-objects. Now let $\alpha \colon H_0 \rightarrow H_1$ be a morphism in $\Lex_{\Sigma}$ with the property that $L(\alpha)$ is an isomorphism. By Lemma~\ref{lemma:flat_lan} this is equivalent to
\begin{equation}\label{eqn:explicit_iso}
 \colim^{(i,j) \in \ca{I} \times \ca{J}} \alpha_{a_i,b_j} \colon \colim^{(i,j) \in \ca{I} \times \ca{J}} H_0(a_i,b_j) \rightarrow \colim^{(i,j) \in \ca{I} \times \ca{J}} H_1(a_i,b_j)
\end{equation}
 being an isomorphism. 

 By definition of $\Sigma$, $H_k(a,-)$ lies in $\Ind(\ca{B})$ for every $a \in \ca{A}$ and $H_k(-,b)$ lies in $\Ind(\ca{A})$ for every $b \in \ca{B}$. Since $\Ind(\ca{A})=\Lex[\ca{A}^{\op},\Sigma]$ is closed under filtered colimits, it follows that
\[
 \colim^{j \in \ca{J}} H_k(-,b_j) \in \Ind(\ca{A})
\]
 for $k=0,1$. From the natural isomorphism
\begin{align*}
 \Lan_Y Uw \bigl(  \colim^{j \in \ca{J}} H_k(-,b_j) \bigr)
 & \cong \colim^{i \in \ca{I}} \colim^{j \in \ca{J}} H_k(a_i,b_j) \\
 & \cong \colim^{(i,j) \in \ca{I} \times \ca{J}} H_k(c_i,d_j) 
\end{align*}
 (see Lemma~\ref{lemma:flat_lan}) and from~\eqref{eqn:explicit_iso} it follows that $\Lan_Y Uw$ sends the morphism
\[
 \colim^{j\in \ca{J}} \alpha_{-,b_j} \colon  \colim^{j \in \ca{J}} H_0(-,b_j) \rightarrow  \colim^{j \in \ca{J}} H_1(-,b_j)
\]
 in $\Ind(\ca{A})$ to an isomorphism. Thus $\colim^{j\in \ca{J}} \alpha_{-,b_j}$ is an isomorphism in $\Ind(\ca{A})$. Since a natural transformation is an isomorphism if and only if all its components are, it follows that
\[
 \colim^{j\in \ca{J}} \alpha_{a,b_j} \colon  \colim^{j \in \ca{J}} H_0(a,b_j) \rightarrow  \colim^{j \in \ca{J}} H_1(a,b_j)
\]
 is an isomorphism for every $a \in \ca{A}$. Using Lemma~\ref{lemma:flat_lan} again, we deduce that
\[
 \Lan_Y Uv (\alpha_{a,-}) \colon \Lan_Y Uv \bigl(H_0(a,-) \bigr) \rightarrow \Lan_Y Uv \bigl( H_1(a,-) \bigr) 
\]
 is an isomorphism for every $a \in \ca{A}$. Since the restriction of $\Lan_Y Uv$ to $\Ind(\ca{B})$ is conservative, it follows that $\alpha_{a,-}$ is an isomorphism, hence that $\alpha_{a,b}$ is invertible for every $(a,b) \in \ca{A}\otimes \ca{B}$. This concludes the proof that $L$ is conservative.
\end{proof}

\begin{lemma}\label{lemma:tensor_product_indabelian}
 The category $\Lex_{\Sigma}$ is locally finitely presentable and abelian, and $\ca{A}\boxtimes \ca{B}$ coincides with the full subcategory of finitely presentable objects. Moreover, for every object $X \in \ca{A}\boxtimes \ca{B}$ there exists a finite family of objects $(A_i,B_j) \in \ca{A}\otimes \ca{B}$ together with an epimorphism
\[
 \oplus_{i,j} Z(A_i,B_j) \rightarrow X \smash{\rlap{.}}
\]
\end{lemma}

\begin{proof}
 The reflections of the representable functors in $\Lex_{\Sigma}$ form a strong generator by the Yoneda lemma. Since filtered colimits commute with finite limits, $\Lex_{\Sigma}$ is closed under filtered colimits in the category of all presehaves on $\ca{A}\otimes \ca{B}$. It follows that the reflections of the representable functors are finitely presentable, hence that $\Lex_{\Sigma}$ is a locally finitely presentable category. Since $L$ is exact and conservative (see Lemma~\ref{lemma:L_lex_conservative}), $\Lex_{\Sigma}$ is also abelian. 

 Since the $Z(A,B)$ form a strong generator, for every $X \in \Lex_{\Sigma}$ there exists an epimorphism 
\[
  \oplus_{i,j} Z(A_i,B_j) \rightarrow X
\]
 where $(A_i,B_j)$ is a (possibly infinite) family of objects in $\ca{A}\otimes \ca{B}$. If $X$ is finitely presentable, the restriction to some finite subfamily is still an epimorphism. Since the kernel of this epimorphism is finitely generated (see Lemma~\ref{lemma:finitely_generated}), there exists an epimorphism onto it from some finitely presentable object $X^{\prime}$. By applying the same argument to $X^{\prime}$ we find that $X$ is in the closure of the $Z(A,B)$ under finite colimits. This shows that finitely presentable objects of $\Lex_{\Sigma}$ lie in $\ca{A}\boxtimes \ca{B}$. The converse inclusion follows from the fact that finitely presentable objects are closed under finite colimits.
\end{proof}

\begin{thm}\label{thm:tensor_weakly_tannakian}
 The category $\ca{A} \boxtimes \ca{B}$ is weakly Tannakian, and the composite
\begin{equation}\label{eqn:fiber_functor} 
 \xymatrix{\ca{A} \boxtimes \ca{B} \ar[r]^-{w\boxtimes v} & \Mod_A^{\fp} \boxtimes \Mod_B^{\fp} \ar[r]^-{\simeq} & \Mod_{A\otimes B}^{\fp}}
\end{equation}
 is a fiber functor.
\end{thm}

\begin{proof}
 The category $\ca{A}\boxtimes \ca{B}$ is ind-abelian by Lemma~\ref{lemma:tensor_product_indabelian}. Thus it suffices to show that the composite~\eqref{eqn:fiber_functor} is a fiber functor and that every object in $\ca{A}\boxtimes \ca{B}$ admits an epimorphism from an object with a dual.

 We first show that the composite~\eqref{eqn:fiber_functor} is a fiber functor. We have already observed that it is strong symmetric monoidal, so it only remains to check that it is faithful and flat. By Lemma~\ref{lemma:fiber_functor_lan} it suffices to check that $LJ$ is faithful and flat. Since $L \colon \Ind(\ca{A}\boxtimes \ca{B}) \rightarrow \Mod_R$ is a conservative and exact left adjoint (see Lemma~\ref{lemma:L_lex_conservative}), this follows from the characterization of flat functors on ind-abelian categories in Proposition~\ref{prop:flat_for_ind_abelian}.

 It remains to show that every object $X \in \ca{A}\boxtimes \ca{B}$ admits an epimorphism from an object with a dual. Since we already know that there is an epimorphism
\[
 \oplus_{i,j} Z(A_i,B_j) \rightarrow X
\]
 for some finite family of objects $(A_i,B_j) \in \ca{A}\otimes \ca{B}$ (see Lemma~\ref{lemma:tensor_product_indabelian}), it suffices to check this for $X=Z(A,B)$. By Corollary~\ref{cor:Z_strong_monoidal}, $Z$ is strong monoidal, so it preserves objects with duals. Therefore it suffices to give a morphism 
\[
(A^{\prime},B^{\prime}) \rightarrow (A,B) 
\]
 in $\ca{A}\otimes \ca{B}$ whose domain has a dual and whose image under $Z$ is an epimorphism.

 To do this, choose epimorphisms $p\colon A^{\prime} \rightarrow A$ and $q \colon B^{\prime} \rightarrow B$ such that $A^{\prime}$ and $B^{\prime}$ have a dual. Then $(A^{\prime},B^{\prime}) \cong (A^{\prime},I)\otimes (I,B^{\prime})$ has a dual, and $Z$ sends the morphism
\[
 \xymatrix{(A^{\prime},B^{\prime}) \ar[r]^-{f\otimes \id_{B^{\prime}}} & (A,B^{\prime}) \ar[r]^-{\id_{A} \otimes q} & (A,B)}
\]
 to an epimorphism because, by construction, $Z$ sends sequences in $\Sigma$ to right exact sequences.
\end{proof}

 \subsection{Right exact tensor products of categories of quasi-coherent sheaves}
 We are now ready to prove that the category of finitely presentable quasi-coherent sheaves on the product of two Adams stacks is given by the right exact tensor product of the categories of finitely presentable quasi-coherent sheaves on the two factors.

\begin{proof}[Proof of Theorem~\ref{thm:stack_product_to_coproduct}]
 We follow the strategy outlined in \S \ref{section:deligne_tensor_intro}. Let $X$ and $Y$ be Adams stacks. The contravariant pseudofunctor
\[
 \QCoh_{\fp}(-) \colon \ca{AS}^{\op} \rightarrow \ca{RM}
\]
 from Adams stacks to right exact symmetric monoidal categories is an embedding by \cite[Theorem~1.3.3]{SCHAEPPI_STACKS}. Since Adams stacks are closed under products (see Theorem~\ref{thm:product_adams}), we know that $X\times Y$ is an Adams stack. Thus $\QCoh_{\fp}(X\times Y)$ is a bicategorical coproduct in the image of $\QCoh_{\fp}(-)$.

 From Theorem~\ref{thm:rex_tensor_coproduct} we know that $\QCoh_{\fp}(X)\boxtimes \QCoh_{\fp}(Y)$ is a bicategorical coproduct in $\ca{RM}$. Thus, if we can show that it lies in the image of $\QCoh_{\fp}(-)$, it must be equivalent to $\QCoh_{\fp}(X\times Y)$.

 In Theorem~\ref{thm:tensor_weakly_tannakian} we proved that weakly Tannakian ind-abelian categories are closed under the right exact tensor product. The conclusion therefore follows from the Tannakian recognition theorem (see Theorem~\ref{thm:recognition}).
\end{proof}

\begin{rmk}
 The above proof of Theorem~\ref{thm:stack_product_to_coproduct} shows that the contravariant pseudofunctor
\[
 \QCoh_{\fp}(-) \colon \ca{AS}^{\op} \rightarrow \ca{RM}
\]
 from Adams stacks to right exact symmetric monoidal categories sends finite products to finite coproducts.
\end{rmk}

 The fact that the category of coherent sheaves on the product of two noetherian schemes with the resolution property is given by the Deligne tensor product of the categories of coherent sheaves of the two factors follows immediately from Theorem~\ref{thm:stack_product_to_coproduct}.

\begin{proof}[Proof of Theorem~\ref{thm:scheme_product_to_coproduct}]
 Let $X$ and $Y$ be noetherian schemes with the resolution property, and assume that $X \times Y$ is noetherian. From Theorem~\ref{thm:stack_product_to_coproduct}, we know that there is an equivalence
\[
 \QCoh_{\fp}(X\times Y) \simeq \QCoh_{\fp}(X) \boxtimes \QCoh_{\fp}(Y)
\]
 of symmetric monoidal $R$-linear categories, where $\boxtimes$ denotes the right exact tensor product. The assumption that $X \times Y$ is noetherian implies that $\QCoh_{\fp}(X\times Y)=\Coh(X\times Y)$ is abelian. Therefore the right exact tensor product satisfies the defining universal property of the Deligne tensor product.
\end{proof}

\appendix
\section{Proof of Theorem~\ref{thm:pseudomonoid_coproduct}} \label{section:pseudomonoid_coproduct}
 Theorem~\ref{thm:pseudomonoid_coproduct} is a categorification of the fact that the tensor product of commutative algebras is a coproduct in the category of commutative algebras, and it is proved in essentially the same way. Fix a pseudomonoid $(A,p,i)$ in $\ca{M}$. We will show that the functor
\begin{equation}\label{eqn:combine}
 \SymPsMon(A_0,A) \times  \SymPsMon(A_1,A) \rightarrow \SymPsMon(A_0 \otimes A_1,A) 
\end{equation}
 which sends $(f,g)$ to $p \cdot f\otimes g$ and the functor
\begin{equation}\label{eqn:split}
\SymPsMon(A_0 \otimes A_1,A) \rightarrow \SymPsMon(A_0,A) \times  \SymPsMon(A_1,A) 
\end{equation}
 which sends $h$ to $(h \cdot A_0 \otimes i_1, h\cdot i_0 \otimes A_1)$ are well-defined and mutually inverse equivalences. To see that they are well-defined, we have to show that $p \colon A\otimes A \rightarrow A$, $A_0\otimes i_1$, and $i_0 \otimes A_1$ are symmetric strong monoidal.

 By the coherence theorem for monoidal bicategories \cite{GORDON_POWER_STREET}, it suffices to prove this for symmetric monoidal Gray monoids in the sense of \cite{DAY_STREET}. In order to do that we use the string diagram notation for Gray monoids introduced in \cite[\S 11.1]{SCHAEPPI}.

 The first observation we use is that
\[
 \otimes \colon \ca{M} \times \ca{M} \rightarrow \ca{M}
\]
 is a symmetric monoidal pseudofunctor, that is, a sylleptic monoidal pseudofunctor between symmetric monoidal 2-categories (see \cite[Definition~15]{DAY_STREET}). Giving a symmetric pseudomonoid in $\ca{M}$ amounts to giving a symmetric monoidal pseudofunctor
\[
 \ast \rightarrow \ca{M} \smash{\rlap{.}}
\]
 Thus $\otimes$ lifts to a pseudofunctor between the respective 2-categories of symmetric pseudomonoids (see \cite[Proposition~16]{DAY_STREET}). Unraveling the definitions (and using the notation from \cite{DAY_STREET}), the product on $A\otimes A$ is given by the composite
\[
 \xymatrix{(A\otimes A) \otimes A\otimes A \ar[rr]^-{A \otimes \rho_{A,A} \otimes A} && (A\otimes A)\otimes (A\otimes A) \ar[r]^-{p \otimes p} & A\otimes A} \smash{\rlap{,}}
\]
 and the symmetry is given by
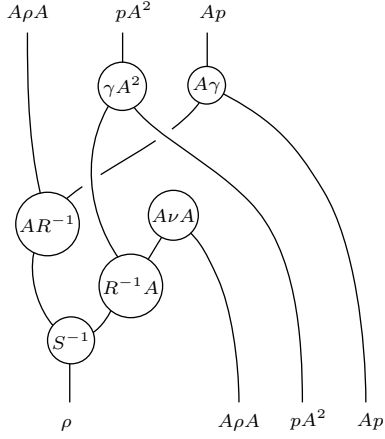
\[
\begin{tikzpicture}[y=0.80pt,x=0.80pt,yscale=-1, inner sep=0pt, outer sep=0pt, every text node part/.style={font=\scriptsize} ]
\path[draw=black,line join=miter,line cap=butt,line width=0.500pt]
  (135.0000,622.3622) .. controls (135.0000,632.3622) and (135.0000,637.3622) ..
  (135.0000,647.3622);
\path[draw=black,line join=miter,line cap=butt,line width=0.500pt]
  (175.0000,622.3622) .. controls (175.0000,632.3622) and (175.0000,637.3622) ..
  (175.0000,647.3622);
\path[draw=black,line join=miter,line cap=butt,line width=0.500pt]
  (175.0000,647.3622) .. controls (205.0000,662.3622) and (215.1517,672.6149) ..
  (230.0000,697.3622) .. controls (244.8483,722.1094) and (250.0000,767.3622) ..
  (250.0000,797.3622);
\path[draw=black,line join=miter,line cap=butt,line width=0.500pt]
  (90.0000,622.3622) .. controls (90.0000,647.3622) and (90.0000,682.3622) ..
  (100.0000,712.3622);
  \path[draw=black,line join=miter,line cap=butt,line width=0.500pt]
    (135.0000,647.3622) .. controls (140.0000,667.3622) and (172.5000,679.2372) ..
    (195.0000,702.3622) .. controls (217.5000,725.4872) and (220.0000,767.3622) ..
    (220.0000,797.3622);
  \path[draw=black,line join=miter,line cap=butt,line width=0.500pt]
    (175.0000,647.3622) .. controls (173.5445,654.6397) and (167.0046,661.4936) ..
    (158.3404,668.0470)(150.6157,673.4921) .. controls (142.2852,679.0281) and
    (132.9848,684.3649) .. (124.6350,689.5825)(115.9614,695.2548) .. controls
    (107.6455,701.0330) and (101.4162,706.6976) .. (100.0000,712.3622);
  \path[draw=black,line join=miter,line cap=butt,line width=0.500pt]
    (131.5057,651.9388) .. controls (111.0262,681.3167) and (121.0012,718.3633) ..
    (140.0000,737.3622);
\path[draw=black,line join=miter,line cap=butt,line width=0.500pt]
  (100.0000,712.3622) .. controls (84.5535,729.0052) and (94.8244,760.1781) ..
  (110.0000,767.3622);
\path[draw=black,line join=miter,line cap=butt,line width=0.500pt]
  (140.0000,737.3622) .. controls (130.0000,747.3622) and (130.0000,762.3622) ..
  (115.0000,767.3622);
\path[draw=black,line join=miter,line cap=butt,line width=0.500pt]
  (160.0000,712.3622) .. controls (155.0000,712.3622) and (145.0000,737.3622) ..
  (140.0000,737.3622);
\path[draw=black,line join=miter,line cap=butt,line width=0.500pt]
  (160.0000,712.3622) .. controls (170.0000,717.3622) and (173.7500,724.8622) ..
  (180.0000,738.6122) .. controls (186.2500,752.3622) and (190.0000,777.3622) ..
  (190.0000,797.3622);
\path[draw=black,line join=miter,line cap=butt,line width=0.500pt]
  (110.0000,767.3622) .. controls (110.0000,777.3622) and (110.0000,787.3622) ..
  (110.0000,797.3622);
\path[fill=black] (99.500023,712.95093) node[circle, draw, line width=0.500pt,
  minimum width=5mm, fill=white, inner sep=0.25mm] (text3983) {$A R^{-1}$    };
\path[fill=black] (110.6117,768.00421) node[circle, draw, line width=0.500pt,
  minimum width=5mm, fill=white, inner sep=0.25mm] (text3987) {$S^{-1}$    };
\path[fill=black] (138.89597,742.24536) node[circle, draw, line width=0.500pt,
  minimum width=5mm, fill=white, inner sep=0.25mm] (text3991) {$R^{-1}A$    };
\path[fill=black] (158.98985,708.82666) node[circle, draw, line width=0.500pt,
  minimum width=5mm, fill=white, inner sep=0.25mm] (text3995) {$A \nu A$    };
\path[fill=black] (134.85536,647.79608) node[circle, draw, line width=0.500pt,
  minimum width=5mm, fill=white, inner sep=0.25mm] (text3999) {$\gamma A^2$
  };
\path[fill=black] (174.75639,647.29102) node[circle, draw, line width=0.500pt,
  minimum width=5mm, fill=white, inner sep=0.25mm] (text4003) {$A \gamma$    };
\path[fill=black] (80.054588,617.23895) node[above right] (text11288) {$A \rho
  A$     };
\path[fill=black] (131.57237,617.23895) node[above right] (text11292) {$p A^2$
  };
\path[fill=black] (171.22086,617.23895) node[above right] (text11296) {$A p$
  };
\path[fill=black] (105.81348,810.93573) node[above right] (text11300) {$\rho$
  };
\path[fill=black] (179.80714,810.68323) node[above right] (text11304) {$A \rho
  A$     };
\path[fill=black] (214.15234,810.43066) node[above right] (text11308) {$p A^2$
  };
\path[fill=black] (245.7196,810.43066) node[above right] (text11312) {$A p$   };

\end{tikzpicture}
\]
 where $R$ and $S$ denote the modifications which replace the hexagon axioms in a braided monoidal 2-category. The multiplication of a commutative abelian group is a group homomorphism, so it does not seem unreasonable to expect that the product $p \colon A\otimes A \rightarrow A$ and the unit $i \colon I \rightarrow A$ are strong symmetric monoidal morphisms, and that the coherence 2-cells are symmetric monoidal natural transformations. The first fact would for example follow if we could show that a pseudofunctor $T \colon \ca{M} \rightarrow \ca{N}$ is sylleptic monoidal if and only if
\[
 \xymatrix{\ca{M} \times \ca{M} \dtwocell\omit{^<-6> \chi} \ar[r]^-{\otimes} \ar[d]_{T\times T} & \ca{M} \ar[d]^{T} \\ 
\ca{N} \times \ca{N} \ar[r]_-{\otimes} & \ca{N}}
\]
 is braided monoidal. Unfortunately, on \cite[p.~125]{DAY_STREET}, it is only proved that $\chi$ is monoidal, not that it is braided. Instead of proving this result in full generality we outline how to check that it holds in the particular case we are interested in. This is greatly simplified by Lack's coherence theorem \cite[Theorem~3.5 and Remark~3.6]{LACK_PSEUDOMONADS} for (non-symmetric) pseudomonoids, which states that any two 2-cells built from associators and unit isomorphisms coincide if they have the same source and target.

\begin{lemma}\label{lemma:product_symmetric_monoidal}
 Let $(A,p,i)$ be a symmetric pseudomonoid. Then the 2-cells
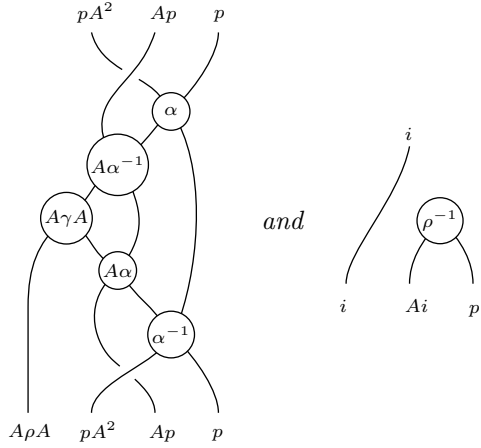
\[
\vcenter{\hbox{
\begin{tikzpicture}[y=0.80pt,x=0.80pt,yscale=-1, inner sep=0pt, outer sep=0pt, every text node part/.style={font=\scriptsize} ]
\path[draw=black,line join=miter,line cap=butt,line width=0.500pt]
  (435.0000,607.3622) .. controls (435.0000,617.3622) and (425.0000,632.3622) ..
  (415.0000,642.3622);
\path[draw=black,line join=miter,line cap=butt,line width=0.500pt]
  (409.2857,647.3622) .. controls (402.7402,656.0091) and (396.6294,662.5800) ..
  (390.0001,669.5052);
\path[draw=black,line join=miter,line cap=butt,line width=0.500pt]
  (415.7143,649.5050) .. controls (430.0000,677.3622) and (425.0000,722.3622) ..
  (415.0000,747.3622);
\path[draw=black,line join=miter,line cap=butt,line width=0.500pt]
  (385.0000,672.3622) .. controls (375.0000,677.3622) and (375.0000,687.3622) ..
  (365.0000,692.3622);
\path[draw=black,line join=miter,line cap=butt,line width=0.500pt]
  (365.0000,697.3622) .. controls (375.0000,702.3622) and (375.0000,712.3622) ..
  (385.0000,717.3622);
\path[draw=black,line join=miter,line cap=butt,line width=0.500pt]
  (390.0000,722.3622) .. controls (395.0000,732.3622) and (405.0000,737.3622) ..
  (410.0000,747.3622);
\path[draw=black,line join=miter,line cap=butt,line width=0.500pt]
  (415.0000,752.3622) .. controls (425.0000,762.3622) and (435.0000,777.3622) ..
  (435.0000,787.3622);
\path[draw=black,line join=miter,line cap=butt,line width=0.500pt]
  (390.0000,672.3622) .. controls (400.0000,692.3622) and (400.0000,702.3622) ..
  (390.0000,717.3622);
\path[draw=black,line join=miter,line cap=butt,line width=0.500pt]
  (410.0000,752.3622) .. controls (410.0000,762.3622) and (375.0000,772.3622) ..
  (375.0000,787.3622)(385.0000,722.3622) .. controls (370.0000,737.3622) and
  (377.6164,755.5816) .. (385.0000,762.3622) .. controls (386.0353,763.3129) and
  (387.1734,764.2919) .. (388.3653,765.3002)(395.4652,771.4280) .. controls
  (400.5937,776.1646) and (405.0000,781.4524) .. (405.0000,787.3622);
\path[draw=black,line join=miter,line cap=butt,line width=0.500pt]
  (410.0000,642.3622) .. controls (408.0475,636.5046) and (403.0450,632.9344) ..
  (397.3746,629.5674)(389.4943,624.9326) .. controls (383.0781,620.9310) and
  (377.1647,616.0210) .. (375.0000,607.3622)(385.0000,667.3622) .. controls
  (379.4230,657.2404) and (377.1164,647.8788) .. (385.0000,637.3622) .. controls
  (392.8836,626.8456) and (400.0000,622.3622) .. (405.0000,607.3622);
\path[draw=black,line join=miter,line cap=butt,line width=0.500pt]
  (360.0000,697.3622) .. controls (350.0000,707.3622) and (345.0000,717.3622) ..
  (345.0000,737.3622) .. controls (345.0000,757.3622) and (345.0000,772.3622) ..
  (345.0000,787.3622);
\path[fill=black] (412.64731,644.76563) node[circle, draw, line width=0.500pt,
  minimum width=5mm, fill=white, inner sep=0.25mm] (text13323) {$\alpha$    };
\path[fill=black] (387.39349,670.01947) node[circle, draw, line width=0.500pt,
  minimum width=5mm, fill=white, inner sep=0.25mm] (text13327) {$A \alpha^{-1}$
  };
\path[fill=black] (387.39349,720.02197) node[circle, draw, line width=0.500pt,
  minimum width=5mm, fill=white, inner sep=0.25mm] (text13331) {$A \alpha$    };
\path[fill=black] (412.64731,750.3266) node[circle, draw, line width=0.500pt,
  minimum width=5mm, fill=white, inner sep=0.25mm] (text13335) {$\alpha^{-1}$
  };
\path[fill=black] (363.14984,695.27325) node[circle, draw, line width=0.500pt,
  minimum width=5mm, fill=white, inner sep=0.25mm] (text13339) {$A \gamma A$
  };
\path[fill=black] (432.09274,800.5816) node[above right] (text13343) {$p$     };
\path[fill=black] (402.04071,800.83417) node[above right] (text13347) {$A p$
  };
\path[fill=black] (369.46329,800.83423) node[above right] (text13351) {$p A^2$
  };
\path[fill=black] (335.62317,800.58167) node[above right] (text13355) {$A \rho
  A$     };
\path[fill=black] (368.20059,602.8443) node[above right] (text13359) {$p A^2$
  };
\path[fill=black] (402.54578,602.8443) node[above right] (text13363) {$A p$
  };
\path[fill=black] (432.85037,602.33923) node[above right] (text13367) {$p$   };

\end{tikzpicture}
}}
\quad \text{and} \quad
\vcenter{\hbox{
\begin{tikzpicture}[y=0.80pt,x=0.80pt,yscale=-1, inner sep=0pt, outer sep=0pt, every text node part/.style={font=\scriptsize} ]
\path[draw=black,line join=miter,line cap=butt,line width=0.500pt]
  (495.0000,802.3622) .. controls (495.0000,787.3622) and (520.0000,762.3622) ..
  (525.0000,737.3622);
\path[draw=black,line join=miter,line cap=butt,line width=0.500pt]
  (525.0000,802.3622) .. controls (525.0000,792.3622) and (530.0000,782.3622) ..
  (540.0000,772.3622);
\path[draw=black,line join=miter,line cap=butt,line width=0.500pt]
  (540.0000,772.3622) .. controls (550.0000,782.3622) and (555.0000,792.3622) ..
  (555.0000,802.3622);
\path[fill=black] (539.42145,772.54993) node[circle, draw, line width=0.500pt,
  minimum width=5mm, fill=white, inner sep=0.25mm] (text14122) {$\rho^{-1}$
  };
\path[fill=black] (522.75397,734.16412) node[above right] (text14126) {$i$
  };
\path[fill=black] (491.94427,815.98651) node[above right] (text14130) {$i$
  };
\path[fill=black] (522.75397,815.98651) node[above right] (text14134) {$A i$
  };
\path[fill=black] (553.05853,817.98651) node[above right] (text14138) {$p$
  };

\end{tikzpicture}
}}
\]
 endow the morphism $p \colon A \otimes A \rightarrow A$ with the structure of a symmetric monoidal morphism.
\end{lemma}

\begin{proof}
 The fact that the 2-cells in question endow $p$ with the structure of a (not necessarily symmetric) monoidal morphism is a consequence of the description of braided monoidal pseudofunctors given in \cite[p.~125]{DAY_STREET}. It only remains to check that $p$ is symmetric. This follows by first applying the defining equation
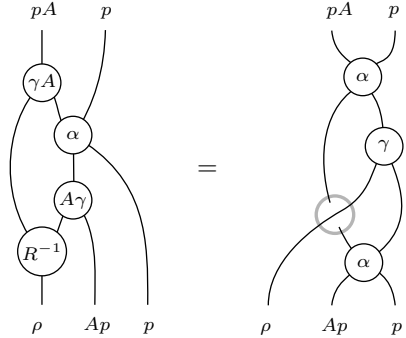
\[
 \vcenter{\hbox{
 \begin{tikzpicture}[y=0.80pt,x=0.80pt,yscale=-1, inner sep=0pt, outer sep=0pt, every text node part/.style={font=\scriptsize} ]
\path[draw=black,line join=miter,line cap=butt,line width=0.500pt]
  (105.0000,872.3622) .. controls (105.0000,882.3622) and (105.0000,887.3622) ..
  (105.0000,897.3622);
\path[draw=black,line join=miter,line cap=butt,line width=0.500pt]
  (135.0000,872.3622) .. controls (134.5382,890.8735) and (130.0000,907.3622) ..
  (120.0000,922.3622);
\path[draw=black,line join=miter,line cap=butt,line width=0.500pt]
  (105.0000,897.3622) .. controls (115.0000,902.3622) and (110.0000,917.3622) ..
  (120.0000,922.3622);
\path[draw=black,line join=miter,line cap=butt,line width=0.500pt]
  (105.0000,897.3622) .. controls (85.0000,917.3622) and (85.0000,957.3622) ..
  (105.0000,977.3622);
\path[draw=black,line join=miter,line cap=butt,line width=0.500pt]
  (120.0000,922.3622) .. controls (119.8089,931.9635) and (120.1725,942.7611) ..
  (120.0000,952.3622);
\path[draw=black,line join=miter,line cap=butt,line width=0.500pt]
  (120.0000,952.3622) .. controls (112.6438,958.4805) and (113.8741,972.9993) ..
  (105.0000,977.3622);
\path[draw=black,line join=miter,line cap=butt,line width=0.500pt]
  (105.0000,977.3622) .. controls (105.0000,982.3622) and (105.0000,992.3621) ..
  (105.0000,1002.3622);
\path[draw=black,line join=miter,line cap=butt,line width=0.500pt]
  (120.0000,952.3622) .. controls (131.8304,966.2441) and (130.0530,986.0044) ..
  (130.0000,1002.3622);
\path[draw=black,line join=miter,line cap=butt,line width=0.500pt]
  (120.0000,922.3622) .. controls (155.0000,942.3622) and (155.1581,971.4127) ..
  (155.0000,1002.3622);
\path[fill=black] (105.05586,897.30377) node[circle, draw, line width=0.500pt,
  minimum width=5mm, fill=white, inner sep=0.25mm] (text17177) {$\gamma A$    };
\path[fill=black] (105.05586,977.10583) node[circle, draw, line width=0.500pt,
  minimum width=5mm, fill=white, inner sep=0.25mm] (text17181) {$R^{-1}$    };
\path[fill=black] (119.70308,952.86218) node[circle, draw, line width=0.500pt,
  minimum width=5mm, fill=white, inner sep=0.25mm] (text17185) {$A \gamma$    };
\path[fill=black] (119.70308,922.05249) node[circle, draw, line width=0.500pt,
  minimum width=5mm, fill=white, inner sep=0.25mm] (text17189) {$\alpha$    };
\path[fill=black] (100.0051,1015.9967) node[above right] (text17493) {$\rho$
  };
\path[fill=black] (99.673332,866.99921) node[above right] (text17517) {$p A$
  };
\path[fill=black] (132.83505,866.99921) node[above right] (text17521) {$p$
  };
\path[fill=black] (124.69257,1015.8487) node[above right] (text17525) {$A p$
  };
\path[fill=black] (153.03812,1015.9967) node[above right] (text17529) {$p$   };

\end{tikzpicture}
 }}
 \quad = \quad
 \vcenter{\hbox{
 \begin{tikzpicture}[y=0.80pt,x=0.80pt,yscale=-1, inner sep=0pt, outer sep=0pt, every text node part/.style={font=\scriptsize} ]
\path[draw=black,line join=miter,line cap=butt,line width=0.500pt]
  (250.0000,872.3622) .. controls (251.5765,887.9470) and (264.6578,886.5904) ..
  (265.0000,897.3622);
\path[draw=black,line join=miter,line cap=butt,line width=0.500pt]
  (280.0000,872.3622) .. controls (279.6794,893.7527) and (267.0947,881.7834) ..
  (265.0000,897.3622);
\path[draw=black,line join=miter,line cap=butt,line width=0.500pt]
  (265.0000,897.3622) .. controls (274.2516,906.8148) and (273.7962,915.0515) ..
  (275.0000,927.3622);
\path[draw=black,line join=miter,line cap=butt,line width=0.500pt]
  (275.0000,927.3622) .. controls (269.7678,933.2358) and (270.3490,948.6431) ..
  (255.0000,957.3622) .. controls (239.3369,966.2597) and (220.0421,981.9979) ..
  (220.0000,1002.3622)(265.0000,897.3622) .. controls (245.4573,907.1335) and
  (243.0941,932.2562) .. (249.2160,953.7869)(253.4467,965.2844) .. controls
  (256.5823,972.1147) and (260.5571,978.0768) .. (265.0000,982.3622);
\path[draw=black,line join=miter,line cap=butt,line width=0.500pt]
  (275.0000,927.3622) .. controls (284.7115,948.3462) and (290.0000,967.3622) ..
  (265.0000,982.3622);
\path[draw=black,line join=miter,line cap=butt,line width=0.500pt]
  (265.0000,982.3622) .. controls (258.1546,989.0290) and (251.4129,991.3235) ..
  (250.0000,1002.3622);
\path[draw=black,line join=miter,line cap=butt,line width=0.500pt]
  (265.0000,982.3622) .. controls (275.0000,987.3622) and (278.7813,992.0164) ..
  (280.0000,1002.3622);
\path[fill=black] (264.65997,894.27332) node[circle, draw, line width=0.500pt,
  minimum width=5mm, fill=white, inner sep=0.25mm] (text14689) {$\alpha$    };
\path[fill=black] (274.76151,927.10327) node[circle, draw, line width=0.500pt,
  minimum width=5mm, fill=white, inner sep=0.25mm] (text14693) {$\gamma$    };
\path[fill=black] (265.16504,982.15656) node[circle, draw, line width=0.500pt,
  minimum width=5mm, fill=white, inner sep=0.25mm] (text14697) {$\alpha$    };
\path[shift={(508.405,-385.23238)},draw=black,miter limit=4.00,draw
  opacity=0.294,line width=1.276pt]
  (-248.0315,1344.6851)arc(0.000:180.000:8.858)arc(-180.000:0.000:8.858) --
  cycle;
\path[fill=black] (216.17264,1015.9967) node[above right] (text17497) {$\rho$
  };
\path[fill=black] (244.96199,1015.9967) node[above right] (text17501) {$Ap$
  };
\path[fill=black] (278.29703,1015.9967) node[above right] (text17505) {$p$
  };
\path[fill=black] (276.7818,865.98901) node[above right] (text17509) {$p$     };
\path[fill=black] (247.48737,866.49408) node[above right] (text17513) {$p A$
  };

\end{tikzpicture}
 }}
\]
 of a symmetric pseudomonoid twice, and then using Lack's coherence theorem. Here the gray circle denotes the pseudonaturality isomorphism of $\rho$.
\end{proof}

\begin{lemma}\label{lemma:unit_symmetric_monoidal}
 Let $(A,p,i)$ be a symmetric pseudomonoid. Then the unique 2-cell
\[
 p \cdot i\times i \cong i
\]
 which exists by Lack's coherence theorem and the identity on $i$ endow $i \colon I \rightarrow A$ with the structure of a symmetric monoidal morphism. Moreover, the natural transformations $\lambda \colon p\cdot i\otimes A \rightarrow \id_A$ and $\rho \colon p \cdot A \otimes i \rightarrow \id$ are monoidal.
\end{lemma}

\begin{proof}
 The morphism in question is monoidal by Lack's coherence theorem. To see that it is symmetric, it suffices to apply the equation
\[
\vcenter{\hbox{
\begin{tikzpicture}[y=0.80pt,x=0.80pt,yscale=-1, inner sep=0pt, outer sep=0pt, every text node part/.style={font=\scriptsize} ]
\path[draw=black,line join=miter,line cap=butt,line width=0.500pt]
  (395.0000,887.3622) .. controls (395.0000,897.3622) and (395.0000,902.3622) ..
  (390.0000,912.3622);
\path[draw=black,line join=miter,line cap=butt,line width=0.500pt]
  (390.4240,911.4459) .. controls (375.5903,918.9175) and (360.0000,927.3622) ..
  (355.0000,942.3622)(365.0000,887.3622) .. controls (365.0000,893.6870) and
  (363.1852,908.8646) .. (367.1456,917.9948)(374.5874,927.4204) .. controls
  (378.5601,932.4181) and (380.0000,936.6702) .. (380.0000,942.3622);
\path[draw=black,line join=miter,line cap=butt,line width=0.500pt]
  (390.0000,912.3622) .. controls (400.0000,917.3622) and (405.6378,936.9342) ..
  (405.0000,942.3622);
\path[shift={(627.9612,-421.67889)},draw=black,miter limit=4.00,draw
  opacity=0.294,line width=1.276pt]
  (-248.0315,1344.6851)arc(0.000:180.000:8.858)arc(-180.000:0.000:8.858) --
  cycle;
\path[fill=black] (392.85715,911.29077) node[circle, draw, line width=0.500pt,
  minimum width=5mm, fill=white, inner sep=0.25mm] (text19145) {$A\gamma$    };
\path[fill=black] (392.85715,883.07648) node[above right] (text19169) {$p$
  };
\path[fill=black] (402.14285,955.93365) node[above right] (text19173) {$p$
  };
\path[fill=black] (375.35715,954.86218) node[above right] (text19177) {$A i$
  };
\path[fill=black] (361.78571,882.00513) node[above right] (text19181) {$i A$
  };
\path[fill=black] (344.64285,956.29077) node[above right] (text19189)
  {$\rho_{I,I}$   };

\end{tikzpicture}
}}
\quad = \quad
\vcenter{\hbox{
\begin{tikzpicture}[y=0.80pt,x=0.80pt,yscale=-1, inner sep=0pt, outer sep=0pt, every text node part/.style={font=\scriptsize} ]
\path[draw=black,line join=miter,line cap=butt,line width=0.500pt]
  (425.0000,872.3622) .. controls (430.0000,882.3622) and (435.0000,882.3622) ..
  (440.0000,892.3622);
\path[draw=black,line join=miter,line cap=butt,line width=0.500pt]
  (440.0000,892.3622) .. controls (445.0000,882.3622) and (450.0000,882.3622) ..
  (455.0000,872.3622);
\path[draw=black,line join=miter,line cap=butt,line width=0.500pt]
  (440.0000,922.3622) .. controls (435.0000,932.3622) and (430.0000,932.3622) ..
  (425.0000,942.3622);
\path[draw=black,line join=miter,line cap=butt,line width=0.500pt]
  (440.0000,922.3622) .. controls (445.0000,932.3622) and (450.0000,932.3622) ..
  (455.0000,942.3622);
\path[fill=black] (440,892.36218) node[circle, draw, line width=0.500pt, minimum
  width=5mm, fill=white, inner sep=0.25mm] (text19149) {$\lambda$    };
\path[fill=black] (439.64285,923.79077) node[circle, draw, line width=0.500pt,
  minimum width=5mm, fill=white, inner sep=0.25mm] (text19153) {$\rho^{-1}$
  };
\path[fill=black] (418.71426,955.57648) node[above right] (text19157) {$Ai$
  };
\path[fill=black] (453.92856,957.21936) node[above right] (text19161) {$p$
  };
\path[fill=black] (452.85715,867.36218) node[above right] (text19165) {$p$
  };
\path[fill=black] (421.78571,865.71936) node[above right] (text19185) {$i A$
  };

\end{tikzpicture}
}}
\]
 from \cite[p.~121]{DAY_STREET}. This equation can also be used to show that $\lambda$ and $\rho$ are monoidal.
\end{proof}

\begin{proof}[Proof of Theorem~\ref{thm:pseudomonoid_coproduct}]
 The functor~\eqref{eqn:combine} is well-defined by Lemma~\ref{lemma:product_symmetric_monoidal}, and the functor~\eqref{eqn:split} is well-defined by Lemma~\ref{lemma:unit_symmetric_monoidal}. To check that \eqref{eqn:combine} and \eqref{eqn:split} are mutually inverse equivalences, we need to construct natural symmetric monoidal isomorphisms between the respective composites and the identity functors. Since a symmetric monoidal 2-cell is simply a monoidal 2-cell between symmetric monoidal 1-cells, it suffices to construct a monoidal isomorphism. 

 The composite of the two canonical 2-cells
\[
 p \cdot h\otimes h \cdot A_0 \otimes i_1 \otimes i_0 \otimes A_1 \cong h \cdot (p_0 \cdot A_0 \otimes i_0) \otimes (p_1 \cdot i_1 \otimes A_1) \cong h
\]
 is monoidal. Indeed, the first of these 2-cells is monoidal since
\[
 \varphi_2 \colon h \colon (p_0\otimes p_1) \cdot A_0 \otimes \rho_{A_1,A_0} \otimes A_1 \Rightarrow p \cdot h \otimes h
\]
 is monoidal (this is a general fact about braided monoidal 1-cells, see \cite[p.~126]{DAY_STREET}). The second 2-cell above is monoidal by Lemma~\ref{lemma:unit_symmetric_monoidal}.
 
 It remains to check that the invertible 2-cells
\[
 p \cdot f\otimes g \cdot (i_0\otimes A_1) \cong p \cdot i \otimes A \cdot g \cong g
\]
 and 
\[
 p \cdot f\otimes g \cdot (A_0 \otimes i_1) \cong p \cdot A \otimes i \cdot f \cong f
\]
 are monoidal. By symmetry, it suffices to consider the first of these, which by Lemma~\ref{lemma:unit_symmetric_monoidal} amounts to showing that $\varphi_0 \colon fi_0 \cong i$ is symmetric monoidal. This is immediate from the definition of the monoidal structure on $i$ and $i_0$ (see Lemma~\ref{lemma:unit_symmetric_monoidal}).
\end{proof}

\bibliographystyle{amsalpha}
\bibliography{indabelian}

\end{document}